\documentclass[12pt,reqno]{amsart}
\usepackage[margin=1in]{geometry}
\makeatletter
\def\section{\@startsection{section}{1}%
	\z@{.7\linespacing\@plus\linespacing}{.5\linespacing}%
	{\bfseries
		\centering
}}
\def\@secnumfont{\bfseries}
\makeatother
\usepackage{amsmath,amssymb,amsthm,graphicx,amsxtra, setspace}
\usepackage[utf8]{inputenc}
\usepackage{mathrsfs}
\usepackage{alltt}
\usepackage{relsize}
\usepackage{hyperref}
\usepackage{aliascnt}
\usepackage{tikz}
\usepackage{mathtools}
\usepackage{multicol}
\usepackage{upgreek}
\usepackage{graphicx,type1cm,eso-pic,color}
\allowdisplaybreaks

\usepackage{pdfrender,xcolor}

\colorlet{darkblue}{blue!50!black}

\hypersetup{
	colorlinks,%
	citecolor=blue,%
	filecolor=red,%
	linkcolor=darkblue,%
	urlcolor=blue,%
	pdfnewwindow=true,%
	pdfstartview={FitH}
}

\newtheorem{theorem}{Theorem}[section]
\newtheorem{lemma}[theorem]{Lemma}

\newtheorem{remark}[theorem]{Remark}

\def\L{\mathrm{L}}

\def\X{\mathrm{X}}

\def\C{\mathrm{C}}

\def\D{\mathrm{D}}

\def\Y{\mathrm{Y}}

\def\Z{\mathrm{Z}}

\def\W{\mathrm{W}}

\def\N{\mathbb{N}}
\def\I{\mathrm{I}}

\def\S{\mathrm{S}}

\def\H{\mathrm{H}}

\newcommand{\R}{\mathbb{R}}

\DeclareMathOperator*{\esssup}{ess\,sup}

\let\originalleft\left
\let\originalright\right
\renewcommand{\left}{\mathopen{}\mathclose\bgroup\originalleft}
\renewcommand{\right}{\aftergroup\egroup\originalright}

\newcommand{\vertiii}[1]{{\left\vert\kern-0.25ex\left\vert\kern-0.25ex\left\vert #1 
		\right\vert\kern-0.25ex\right\vert\kern-0.25ex\right\vert}}

\newcommand{\Addresses}{{
		\footnote{
			
			\noindent \textsuperscript{1,2}Department of Mathematics, Indian Institute of Technology Roorkee-IIT Roorkee,
			Haridwar Highway, Roorkee, Uttarakhand 247667, INDIA.\par\nopagebreak
			\noindent  \textit{e-mail:} \texttt{Manil T. Mohan: maniltmohan@ma.iitr.ac.in, maniltmohan@gmail.com.}
			
			\textit{e-mail:} \texttt{Shri Lal Raghudev Ram Singh: raghudevram$\_$s@ma.iitr.ac.in.}
			
			\noindent \textsuperscript{*}Corresponding author.

			\textit{Key words:} generalized Burgers-Huxley equation, stabilization, boundary control.
			
			Mathematics Subject Classification (2020): 93D15, 93D23, 34H05, 35K51.

	}}}

\begin{document}

	\title[Boundary control of generalized Burgers-Huxley equation]{Analysis and Numerical Study of Boundary control of generalized Burgers-Huxley equation\Addresses}
	\author[S. L. Raghudev Ram Singh and M. T. Mohan]
	{Shri Lal Raghudev Ram Singh\textsuperscript{1} and Manil T. Mohan\textsuperscript{2*}}

	\maketitle
	
	\begin{abstract}
		In this work, a boundary control problem for  the following generalized Burgers-Huxley (GBH) equation: $$u_t=\nu u_{xx}-\alpha u^{\delta}u_x+\beta u(1-u^{\delta})(u^{\delta}-\gamma), $$ where $\nu,\alpha,\beta>0,$ $1\leq\delta<\infty$, $\gamma\in(0,1)$  subject to Neumann boundary conditions is analyzed. Using  the Minty-Browder theorem,  standard elliptic partial differential equations theory, the maximum principle and the Crandall-Liggett theorem, we first address the global existence of a unique strong solution to GBH equation.   Then, for the boundary control problem, we prove that the controlled  GBH equation (that is, the closed loop system) is exponentially stable in the $\H^1$-norm (hence pointwise) when the viscosity $\nu$ is known (non-adaptive control). Moreover, we show that a damped version of GBH equation is globally asymptotically stable (in the $\L^2$-norm), when $\nu$ is unknown (adaptive control). Using the Chebychev collocation method with the backward Euler method as a temporal scheme, numerical findings are reported for both the non-adaptive and adaptive situations, supporting and confirming the analytical results of both the controlled and uncontrolled systems.
	\end{abstract}

	\section{Introduction}\label{sec1}\setcounter{equation}{0}  
	For the past few decades, a good amount of research work has been carried out in the control of fluid flows (cf. \cite{FART,MDG,AVF,MK1,HLPH,SSS}, etc.). Fluid flow separation, combustion, and fluid-structure interactions are only a few of the numerous applications that drove the adoption of control theory.  Burgers' equation is a fundamental partial differential equation (PDE) and a basic convection-diffusion equation which is used in many applied mathematics fields, including gas dynamics, fluid mechanics, nonlinear acoustics, and traffic flow.  The control of Burgers equation, which is also considered as a ``toy model'' of Navier-Stokes (NS) equations, has been thoroughly studied in the works \cite{BAKK,ABMK,CIB,BBK,MK,MK2,KKSV,JML,WJJ,IM,NS,NS1}, etc. The 1D Burgers equation received a lot of interest from the control and mathematical communities as NS equations control problems are difficult  to solve numerically. For numerical studies on the stabilization of viscous Burgers equation, the interested readers are referred to see \cite{SKAK,SKAK1,NS}, etc.	The control problems for stochastic Burgers equation have been explored in the works \cite{HCRT,GDAD,GDAD1,MK2}, etc. 
	
	The boundary control (stabilization) of  viscous as well as  inviscid Burgers equation, using both Neumann and Dirichlet boundary control, is investigated  in \cite{MK,MK2}. Under the proposed nonlinear boundary conditions in \cite{MK,MK2}, the global asymptotic stabilization and semiglobal exponential stabilization in the $\H^1$-sense of viscous Burgers equations is examined in \cite{ABMK}. The adaptive and non-adaptive stabilization of  generalized Burgers equation by nonlinear boundary control is analyzed in the work \cite{NS1}. The boundary stabilization of Burgers equation is considered in \cite{MA} and the author proved that the closed‐loop system is globally $\H^1$- and $\H^3$-stable.  The dynamics of the forced Burgers equation subject to both Neumann boundary conditions and periodic boundary conditions using boundary and distributed control is analyzed in the work \cite{NS}. A simple, finite-dimensional, linear deterministic stabilizing boundary feedback law for the stochastic Burgers equation with unbounded time-dependent coefficients is designed in the work \cite{IM}. Explicit time-varying feedback laws leading to the global (null)	stabilization in small time of the viscous Burgers equation with three scalar controls is constructed in \cite{JMC}. 
	
	\subsection{The model}
	The \emph{generalized Burgers-Huxley (GBH) equation}  describes a prototype model for describing the interaction between reaction mechanisms, convection effects and diffusion transports.  For the boundary control problem (stabilization), we  investigate the following GBH (reaction-diffusion-convection) equation: for $ x\in(0,1),\ t>0$,
	\begin{align}\label{1}
		u_t(x,t)=\nu u_{xx}(x,t)-\alpha u^{\delta}(x,t)u_x(x,t)+\beta u(x,t)(1-u^{\delta}(x,t))(u^{\delta}(x,t)-\gamma), 
	\end{align}
	with $u(x,0)=u_0(x),$	using the following nonlinear Neumann boundary control:
	\begin{align}\label{2}
		u_x(0,t)=w_1(t)\ \text{ and }\ u_x(1,t)=w_2(t). 
	\end{align}
		In \eqref{1}, $\alpha>0$ is the advection coefficient, $\nu,\beta>0,$ $1\leq\delta<\infty$ and $\gamma\in(0,1)$ are parameters. The model \eqref{1} with $\alpha=\delta=1$ and $\beta=0$ is the classical Burgers equation and the case  $\delta=1$ is known as the Burgers-Huxley equation. The global solvability (the existence and uniqueness of global weak as well as strong solution) of the GBH equation \eqref{1} with homogeneous Dirichlet's boundary condition is established in \cite{MTAK}. The global solvability of multidimensional stationary as well as non-stationary GBH equations and their  numerical studies are carried out in \cite{AKMTRB,SMAK,SMAK1}, etc.  Stochastic versions of 1D GBH equation and their qualitative properties  have been examined in the works \cite{AKMTM,AKMTM1,AKMTM2,MTM,MTM1}, etc.  

\subsection{Approaches and novelties} 
We first show the well-posedness of the problem \eqref{1}-\eqref{2} when the boundary data is in the feedback form. 	We follow the works \cite{ABMK,KIYY}, etc. to obtain the required results. Our initial method for the global solvability is based on the maximal principle and the semigroup approach.  The existence of strong solutions is established by showing that the sum of  linear and nonlinear operators appearing in the GBH equation \eqref{1} is \emph{strongly monotone, hemicontinuous and coercive}. We demonstrate  the existence of strong solutions with the aid of the \emph{bijectivity of the operator (Minty-Browder theorem), the standard results from the theory of elliptic PDE,  the maximum principle, and the Crandall-Liggett Theorem}. For the classical solutions, we directly borrow  the results from \cite{OAVA} and employ it to the problem \eqref{1}-\eqref{2}. 

For the stabilization of the GBH equation \eqref{1}, the  non-adaptive ($\nu$ is known) and adaptive ($\nu$ is unknown for a damped version) cases are analyzed in this work. The major difference between adaptive and non-adaptive control is that in adaptive control, good control performance can be directly achieved even in the presence of undesirable or unpredictable disturbances (\cite{NS1}).  We mainly follow the works \cite{MK,NS,NS1}, etc. to reach our major goals.  For the boundary control problem, when the viscosity $\nu$ is known (non-adaptive control), we prove that the controlled  GBH equation (that is, the closed loop system) is exponentially stable not only in the $\L^2$-norm but also in the $\H^1$-norm. With the help of Agmon's inequality, we also obtain the pointwise exponential stability.  

Numerical results are provided for both the adaptive and non-adaptive cases using the Chebychev collocation method with the backward Euler method as a temporal scheme, validating and supporting the analytical results of both the controlled and uncontrolled systems.  	To the best of authors' knowledge, this work appears to be the first one which discusses the stabilization (by boundary control) of GBH equation.  Indeed, this paper's results also hold true for Burger's equation, improving the findings from \cite{ABMK,MK}, etc. 
	
\subsection{Organization of the paper}	
The rest of the paper is organized as follows: In the next section, we consider the GBH equation \eqref{1} subject to \eqref{2} in the feedback form. A discussion on the existence and uniqueness of classical solution is provided in subsection \ref{cls} (Theorem \ref{thm2.2}).  We then establish the global solvability results (Theorem \ref{thm2.6}) by applying   the Minty-Browder theorem (Lemma \ref{lem2.1}),  standard elliptic PDE theory (Lemma \ref{lem2.3}), the maximum principle (Lemma \ref{lem2.6}) and the Crandall-Liggett theorem (Theorem \ref{thm25}).   In section \ref{sec3}, we first show the $\L^2$-exponential stability for the system  \eqref{1}-\eqref{2} under the control law \eqref{eq6} (non-adaptive control), when the viscosity $\nu>\frac{\beta}{2}(1-\gamma)^2$ is known (Theorem \ref{thm2.3}). Agmon's (Lemma \ref{lem21}) and Poincar\'e (Lemma \ref{lem3.1}) inequalities play a crucial role in our analysis.  A discussion on the asymptotic stability of forced GBH equation and Dirichlet's boundary control of unforced GBH equation is presented  in  Theorem \ref{rem2.4} and  Remark \ref{rem2.5}, respectively.  Theorem \ref{thm3.6} yields the $\L^q$-norm exponential stability results for arbitrary $q\in[2,\infty)$. This result  and Agmon's inequality   lead to the global exponential stability in the $\H^1$-norm and pointwise sense, (Theorems \ref{thm3.8} and  \ref{thm3.10}).  Then, we consider the adaptive control problem (that is, when $\nu>0$ is unknown) for a damped version of the system \eqref{1}-\eqref{2}  (see \eqref{1d} below, where $\kappa\leq \frac{\beta}{4}(1-\gamma)^2$ is known) and prove the asymptotic stability under the control law \eqref{29} (Theorem \ref{thm2.5}). In the final section, numerical studies are carried out to validate the theoretical findings obtained in the paper.

\section{Well-posedness}\label{sec2}\setcounter{equation}{0}  
In this section, we discuss the existence and uniqueness of solutions of \eqref{1} subject to \eqref{2} given in the feedback form  \eqref{eq6} or \eqref{2p10}. Similar results for Burgers equation and viscous scalar conservation law are obtained in \cite{ABMK,KIYY}, respectively. The following Agmon's inequality is used in the sequel (see \cite[Lemma A.2]{MK}). 
\begin{lemma}[Agmon's inequality]\label{lem21}
	For any $w\in\C([0,1])$, the following inequality holds: 
	\begin{align}\label{21}
		\max_{x\in[0,1]}|w(x)|\leq C_a\|w\|_{\L^2}^{1/2}\|w\|_{\H^1}^{1/2},
	\end{align}
	where $C_a=2^{5/4}$. 
	\end{lemma}
	\begin{proof}
		An application of H\"older's inequality yields 
		\begin{align}\label{22}
			w^2(x)&=w^2(0)+2\int_0^xw(\zeta)w_x(\zeta)d\zeta\leq w^2(0)+2\left(\int_0^1|w(\zeta)|^2d\zeta\right)^{1/2}\left(\int_0^1|w_x(\zeta)|^2d\zeta\right)^{1/2},
		\end{align}
		for all $x\in[0,1]$. Note that 
		\begin{align*}
			|w^2(0)|&=\left|w^2(x)-2\int_0^xw(\zeta)w_x(\zeta)d\zeta\leq w^2(0)\right|\leq|w(x)|^2+2\|w\|_{\L^2}\|w_x\|_{\L^2}.
		\end{align*}
		Integrating the above inequality from $0$ to $1$, we deduce 
		\begin{align}
			|w^2(0)|\leq \|w\|_{\L^2}(\|w\|_{\L^2}+2\|w_x\|_{\L^2}).
		\end{align}
		Therefore, from \eqref{22}, we infer 
		\begin{align}\label{24}
			\max\limits_{x\in[0,1]}|w(x)|^2\leq \|w\|_{\L^2}(\|w\|_{\L^2}+4\|w_x\|_{\L^2})\leq 4\sqrt{2}\|w\|_{\L^2}\left(\|w\|_{\L^2}^2+\|w_x\|_{\L^2}^2\right)^{1/2},
		\end{align}
		and \eqref{21} follows. 
	\end{proof}
	\begin{remark}\label{rem2.2}
		From the inequality \eqref{21}, it is also clear that $\max\limits_{x\in[0,1]}|w(x)|\leq C_1\|w\|_{\H^1}$, for all $w\in\H^1(0,1)$, so that $\H^1(0,1)\hookrightarrow\C([0,1])$. In fact, $\mathrm{W}^{1,p}(0,1)\hookrightarrow\C([0,1])$ for any $1\leq p<\infty$ and the embedding is compact (\cite[Example 6, page 172]{DMDZ}). 
	\end{remark}

\subsection{Strong solutions}
The existence of a strong  solutions can be proved by following the method used in \cite[Theorem 2.2]{KIYY}. For  Dirichlet's boundary conditions, similar methods have been employed in \cite[Proposition 3.5]{AKMTM1} to obtain the existence of mild solutions of \eqref{1}. We consider  \eqref{1} subject to the following boundary conditions:
\begin{equation}\label{2p6}
	\left\{
\begin{aligned}
	u_x(0,t)&=g_0(u(0,t)),\\
	u_x(1,t)&=-g_1(u(1,t)), 
	\end{aligned}
	\right.
	\end{equation}
	where $g_0,g_1\in \C(\R)$ are nondecreasing functions with $g_0(0)=g_1(0)=0$	representing the nonlinear flux feedback controls. Let us define a nonlinear map $\mathscr{A}:(\D(\mathscr{A})\subset\L^2(0,1))\to \L^2(0,1)$ by 
\begin{align*}
	\mathscr{A}(v)(x):=-\nu v_{xx}(x)+\frac{\alpha}{\delta+1}( v^{\delta+1})_x-\beta v(x)(1-v^{\delta}(x))(v^{\delta}(x)-\gamma)
\end{align*}
with the domain 
\begin{align*}
	\D(\mathscr{A}):=\left\{v\in \H^2(0,1):v_x(0)=g_0(v(0))\ \text{ and }\ v_x(1)=-g_1(v(1))\right\}.
\end{align*}
Performing an integration by parts,  for each $v\in\D(\mathscr{A})$, we have
\begin{align}\label{2.6}
	(\mathscr{A}(v),w)&=\nu(v_x,w_x)+\nu g_0(v(0))w(0)+\nu g_1(v(1))w(1)\nonumber\\&\quad+\frac{\alpha}{\delta+1}\left[v^{\delta+1}(1)w(1)-v^{\delta+1}(0)w(0)-(v^{\delta+1},w_x)\right]\nonumber\\&\quad-\beta (v(1-v^{\delta})(v^{\delta}-\gamma),w),
\end{align}
for all $w\in \H^1(0,1)=:\mathrm{X}$. We apply the well known \emph{Crandall-Liggett theorem} to obtain the existence of a contraction $\C_0$-semigroup for the problem \eqref{1} along with the boundary condition \eqref{2p6}. Due  to the lack of global Lipschitz continuity, to obtain monotonicity of the nonlinear map $\mathscr{A}$, we introduce a cut-off function for the function $f(v)=\frac{1}{\delta+1}v^{\delta+1}$. In fact, the following results are true for any $f\in \C^1(\R)$ and we prove it in general.  For $\rho\geq 1$, we define 
\begin{align*}
	f_{\rho}(y):=\left\{\begin{array}{ll}
		f(\rho)+f'(\rho)(y-\rho)&\ \text{ if }\ y>\rho,\\ 
		f(y)&\ \text{ if }\ |y|\leq\rho,\\
		f(-\rho)+f'(-\rho)(y+\rho)&\ \text{ if }\ y<-\rho. 
		\end{array}\right.
\end{align*}
Then one can show that $f_{\rho}\in\C^1(\R)$ is globally Lipschitz continuous with the Lipschitz constant, $$L_{\rho}=\sup_{x\in\R}|f_{\rho}'(x)|=\sup_{|x|\leq\rho}|f'(x)|.$$  Let us define the corresponding nonlinear map $\mathscr{A}_{\rho}:\mathrm{X}\to\mathrm{X}^*$ by 
\begin{align}\label{25}
	\langle\mathscr{A}_{\rho}(v),w\rangle &=\nu(v_x,w_x)+\nu g_0(v(0))w(0)+\nu g_1(v(1))w(1)\nonumber\\&\quad+\alpha\left[f_{\rho}(v(1))w(1)-f_{\rho}(v(0))w(0)-(f_{\rho}(v),w_x)\right]\nonumber\\&\quad-\beta (v(1-v^{\delta})(v^{\delta}-\gamma),w),
\end{align} 
for all $w\in\mathrm{X}$. Therefore, for each $w\in\X$, we have 
\begin{align}\label{26}
	|\langle\mathscr{A}_{\rho}(v),w\rangle |&\leq C \big[\nu(\|v_x\|_{\mathrm{L}^2}+|g_0(v(0))|+|g_1(v(0))|)\nonumber\\&\quad+\alpha(\|f_{\rho}(v)\|_{\L^2}+|f_{\rho}(v(0))|+|f_{\rho}(v(1))|)\nonumber\\&\quad+\beta((1+\gamma)\|v\|_{\L^{2(\delta+1)}}^{\delta+1}+\gamma\|v\|_{\L^2}+\|v\|_{\L^{2(\delta+1)}}^{2\delta+1})\big]\|w\|_{\H^1},
\end{align}
where we have used the continuous Sobolev embedding  $\H^1(0,1)\hookrightarrow\C([0,1])\hookrightarrow\L^p(0,1)$ for all $1\leq p<\infty$. From \eqref{26}, we infer that the operator $\mathscr{A}_{\rho}(v)$ is a bounded linear functional on $\X$ for each $v\in\X$. 
\begin{lemma}\label{lem2.1}
	Let $g_0,g_1\in\C(\R)$ be non-decreasing functions, $f\in\C^1(\R)$, $\nu, \alpha,\beta,\rho>0$, $\delta\geq 1$, $\gamma\in(0,1)$ be constants, and $\mathscr{A}_{\rho}$ be defined by \eqref{25}. Then for all $\omega\geq\omega_{\rho}$, with $\omega_{\rho}$ defined by 
	\begin{align}\label{2.12}
		\omega_{\rho}=\left(\frac{\alpha^2L_{\rho}^2}{2\nu}+\frac{2\alpha^2L_{\rho}^2C_a^2}{\max\left\{\frac{\nu}{2},\beta\gamma\right\}}+2^{2\delta-1}\beta(1+\gamma)^2(\delta+1)^2\right),
	\end{align}
	 the nonlinear map $\mathscr{B}_{\rho,\omega}=\mathscr{A}_{\rho}+\omega\I:\X\to\X^*$ is \emph{monotone} in the sense that 
	\begin{align}
		\langle \mathscr{B}_{\rho,\omega}(v)-\mathscr{B}_{\rho,\omega}(w),v-w\rangle \geq 0\ \text{ for all }\ v,w\in \X. 
	\end{align}
	Moreover, $\mathscr{B}_{\rho,\omega}$ is \emph{hemicontinuous} in the sense that for all $v_1,v_2,w\in\X$, $$\lim\limits_{\lambda\to 0}\langle\mathscr{B}_{\rho,\omega}(v_1+\lambda v_2),w\rangle=\langle\mathscr{B}_{\rho,\omega}(v_1),w\rangle.$$ Finally, $\mathscr{B}_{\rho,\omega}$ is \emph{coercive} in the sense that 
	$$\lim\limits_{\|v\|_{\H^1\to\infty}}\frac{\langle\mathscr{B}_{\rho,\omega}(v),v\rangle}{\|v\|_{\H^1}}=\infty.$$ 
\end{lemma}
\begin{proof}
	The proof is divided into the following steps:
	\vskip 0.1cm
	\noindent \textbf{Step 1:} \emph{Monotonicity.}
Let us first prove the monotoniciy property. For all $v,w\in\X$, since $f_{\rho}$ has the global Lipschitz constant $L_{\rho}$ and, $g_0$  and $g_1$ are nondecreasing, we have
\begin{align}\label{28}
	&\langle\mathscr{A}_{\rho}(v)-\mathscr{A}_{\rho}(w),v-w\rangle\nonumber\\&=\nu\|v_x-w_x\|_{\L^2}^2-\alpha(f_{\rho}(v)-f_{\rho}(w),v_x-w_x)\nonumber\\&\quad-\beta(1+\gamma)(v^{\delta+1}-w^{\delta+1},v-w)+\beta\gamma\|v-w\|_{\L^2}^2+\beta(v^{2\delta+1}-w^{2\delta+1},v-w)\nonumber\\&\quad+\nu (g_0(v(0))-g_0(w(0)))(v(0)-w(0))+\nu (g_1(v(0))-g_1(w(0)))(v(0)-w(0))\nonumber\\&\quad +\alpha(f_{\rho}(v(1))-f_{\rho}(w(1)))(v(1)-w(1))-\alpha(f_{\rho}(v(0))-f_{\rho}(w(0)))(v(0)-w(0))\nonumber\\&\geq \nu\|v_x-w_x\|_{\L^2}^2-\alpha L_{\rho}\|v-w\|_{\L^2}\|v_x-w_x\|_{\L^2}-2\alpha L_{\rho}\|v-w\|_{\L^{\infty}}^2+\beta\gamma\|v-w\|_{\L^2}^2\nonumber\\&\quad-\beta(1+\gamma)(v^{\delta+1}-w^{\delta+1},v-w)+\beta(v^{2\delta+1}-w^{2\delta+1},v-w). 
\end{align}
Let us now estimate the term $-\beta(v^{2\delta+1}-w^{2\delta+1},v-w)$ from \eqref{28} as
\begin{align}\label{2.9}
	&\beta(v^{2\delta+1}-w^{2\delta+1},v-w)\nonumber\\&= \beta(v^{2\delta},(v-w)^2) +\beta(w^{2\delta},(v-w)^2)+\beta(v^{2\delta}w-w^{2\delta}v,v-w)\nonumber\\&=\beta\|v^{\delta}(v-w)\|_{\L^2}^2+\beta\|w^{\delta}(v-w)\|_{\L^2}^2+\beta(vw,v^{2\delta}+w^{2\delta})-\beta(v^2,w^{2\delta})-\beta(w^2,v^{2\delta})\nonumber\\&=\frac{\beta}{2}\|v^{\delta}(v-w)\|_{\L^2}^2+\frac{\beta}{2}\|w^{\delta}(v-w)\|_{\L^2}^2+\frac{\beta}{2}(v^{2\delta}-w^{2\delta},v^2-w^2)\nonumber\\&\geq \frac{\beta}{2}\|v^{\delta}(v-w)\|_{\L^2}^2+\frac{\beta}{2}\|w^{\delta}(v-w)\|_{\L^2}^2,
\end{align}
since $(v^{2\delta}-w^{2\delta},v^2-w^2)\geq 0$. Using Taylor's formula, H\"older's and Young's inequalities, we estimate the term $\beta(1+\gamma)(v^{\delta+1}-w^{\delta+1},v-w)$ from \eqref{28} as 
\begin{align}\label{2.10}
	&\beta(1+\gamma)(v^{\delta+1}-w^{\delta+1},v-w)\nonumber\\&=\beta(1+\gamma)(\delta+1)((\theta v+(1-\theta)w)^{\delta}(v-w),v-w)\nonumber\\&\leq \beta(1+\gamma)(\delta+1)2^{\delta-1}(\|v^{\delta}(v-w)\|_{\L^2}+\|w^{\delta}(v-w)\|_{\L^2})\|v-w\|_{\L^2}\nonumber\\&\leq\frac{\beta}{4}\|v^{\delta}(v-w)\|_{\L^2}^2+\frac{\beta}{4}\|w^{\delta}(v-w)\|_{\L^2}^2+\frac{\beta}{2}2^{2\delta}(1+\gamma)^2(\delta+1)^2\|v-w\|_{\L^2}^2.
\end{align}
Combining \eqref{2.9}-\eqref{2.10} and substituting it in \eqref{28},  and then using Young's inequality and  Agmon's inequality $\|u\|_{\L^{\infty}}\leq C_a\|u\|_{\L^2}^{1/2}\|u\|_{\H^1}^{1/2}$, for all $u\in\H^1(0,1)$ (Lemma \ref{lem21}), we obtain 
\begin{align}\label{2.11}
	&\langle\mathscr{A}_{\rho}(v)-\mathscr{A}_{\rho}(w),v-w\rangle\nonumber\\&\geq \nu\|v_x-w_x\|_{\L^2}^2-\alpha L_{\rho}\|v-w\|_{\L^2}\|v_x-w_x\|_{\L^2}-2\alpha L_{\rho}\|v-w\|_{\L^{\infty}}^2+\beta\gamma\|v-w\|_{\L^2}^2\nonumber\\&\quad+\frac{\beta}{4}\|v^{\delta}(v-w)\|_{\L^2}^2+\frac{\beta}{4}\|w^{\delta}(v-w)\|_{\L^2}^2-2^{2\delta-1}\beta(1+\gamma)^2(\delta+1)^2\|v-w\|_{\L^2}^2\nonumber\\&\geq \max\left\{\frac{\nu}{2},\beta\gamma\right\}\|v-w\|_{\H^1}^2-\frac{\alpha^2L_{\rho}^2}{2\nu}\|v-w\|_{\L^2}^2-2\alpha L_{\rho}C_a\|v-w\|_{\L^2}\|v-w\|_{\H^1}\nonumber\\&\quad -2^{2\delta-1}\beta(1+\gamma)^2(\delta+1)^2\|v-w\|_{\L^2}^2\nonumber\\& \geq\frac{1}{2}\max\left\{\frac{\nu}{2},\beta\gamma\right\}\|v-w\|_{\H^1}^2-\omega_{\rho}\|v-w\|_{\L^2}^2,
\end{align}
where $\omega_{\rho}$ is defined in \eqref{2.12}. Therefore for all $\omega\geq\omega_{\rho}$, we deduce from \eqref{2.11} that 
\begin{align}\label{2.13}
	\langle\mathscr{A}_{\rho}(v)-\mathscr{A}_{\rho}(w),v-w\rangle+\omega\|v-w\|_{\L^2}^2\geq \frac{1}{2}\max\left\{\frac{\nu}{2},\beta\gamma\right\}\|u-v\|_{\H^1}^2,
\end{align}
for all $v,w\in\X$, so that the monotonicity of the operator $\mathscr{B}_{\rho,\omega}=\mathscr{A}_{\rho}+\omega\I$ follows.

\vskip 0.1cm
\noindent \textbf{Step 2:} \emph{Hemicontinuity.} Since $\X$ is a  reflexive Banach space, the demicontinuity of $\mathscr{B}_{\rho,\omega}$  implies its hemicontinuity. In order to show  the demicontinuity of the operator $\mathscr{B}_{\rho,\omega}$, it is enough to prove the demicontinuity of the operator $\mathscr{A}_{\rho}(\cdot)$ defined in \eqref{25}, that is, we need to show that if $v_n\to v$ in $\X$ implies $\mathscr{A}_{\rho}(v_n)\xrightharpoonup{w}\mathscr{A}_{\rho}(v)$ in $\X^*$ as $n\to\infty$. Moreover, it is sufficient  to consider the case $|v_n|,|v|\leq\rho$ only as the other cases follow in a similar way.   We choose a sequence $\{v_n\}_{n\in\N}\in\X$ such that $v_n\to v$ in $\X$ and $|v_n|,|v|\leq\rho$. Since $\H^1(0,1)\hookrightarrow\C([0,1])$, the convergence $v_n\to v$ in $\X$ implies $v_n(x)\to v(x)$ for all $x\in[0,1]$.   For any $w\in \X$, we consider 
\begin{align}\label{212}
	\left|	\langle \mathscr{A}(v_n)-\mathscr{A}(v),w\rangle \right|&\leq\big|\nu \big((v_{n,x}-v_x),w_x\big)\big|+\nu |(g_0(v_n(0))-g_0(v(0)))w(0)|\nonumber\\&\quad+\nu |(g_1(v_n(1))-g_1(v(1)))w(1)|+\alpha|(f(v_n)-f(v),w_x)|\nonumber\\&\quad+\alpha|(f(v_n(1))-f(v(1)))w(1)|+\alpha|(f(v_n(0))-f(v(0)))w(0)|\nonumber\\&\quad+\beta(1+\gamma)|(v_n^{\delta+1}-v^{\delta+1},w)|+\beta|(v_n^{2\delta+1}-v^{2\delta+1},w)|+\beta\gamma|(v_n-v,w)|\nonumber\\&\to 0\ \text{ as }\ n\to\infty, 
\end{align}so that the operator $\mathscr{A}:\X\to\X^*$ is demicontinuous, which implies that the operator $\mathscr{A}(\cdot)$ is hemicontinuous also.

\vskip 0.1cm
\noindent \textbf{Step 3:} \emph{Coercivity.} By taking $w=0$ in \eqref{2.11}, one can easily see that 
\begin{align}
	\frac{\langle\mathscr{B}_{\rho,\omega}(v),v\rangle}{\|v\|_{\H^1}}\geq \frac{1}{2}\max\left\{\frac{\nu}{2},\beta\gamma\right\}\|v\|_{\H^1}-\|\mathscr{A}_{\rho}(0)\|_{\X^*}\to 0\ \text{ as }\ \|v\|_{\H^1}\to 0, 
\end{align}
which completes the proof of coercivity.
\end{proof}

\begin{remark}
	 Since $\mathscr{B}_{\omega,\rho}$ is monotone, hemicontinuous and coercive from $\X$ to $\X^*$, then by an application of \cite[Example 2.3.7]{HB}, we deduce that $\mathscr{B}_{\rho,\omega}$ is maximal monotone in $\H$. 
\end{remark}

Let $\Y$ be a Banach space and $\mathrm{T}:\D(\mathrm{T})\subset\Y\to\Y^*$ is said to be \emph{strongly monotone} if there exists $\alpha>0$ such that $\langle \mathrm{T}(u)-\mathrm{T}(v),u-v\rangle\geq\alpha\|u-v\|_{\Y}^2,$ for all $u,v\in\D(\mathrm{T})$.   Before proceeding further, let us recall an important result (\cite[Theorem 3.3.1, Corollary 3.3.1]{GDJM}). 
\begin{lemma}[Minty-Browder theorem]
	Let $\Y$ be a reflexive Banach space and $\mathrm{T}:\D(\mathrm{T})\subset\Y\to\Y^*$ be a coercive hemicontinuous monotone map. Then $\mathrm{T}$ is onto. Moreover, if $\mathrm{T}$ is strongly monotone, then $\mathrm{T}$ is a bijection. 
\end{lemma}

From the relation \eqref{2.11}, we deduce that the operator $\mathscr{B}_{\rho,\omega}$ is strongly monotone. By applying the above theorem and the standard elliptic PDE theory (\cite[Theorems 8.8, 8.9, Corollary 8.11]{DGNS}), we immediately have  the following result: 
\begin{lemma}\label{lem2.3}
	Assume that conditions of Lemma \ref{lem2.1} hold true. Then the operator $\mathscr{B}_{\rho, \omega}:\X\to\X^*$ defined in Lemma \ref{lem2.1} is one-one and onto, that is, for each $\xi\in\X^*$, there exists a unique $\varphi\in\X$ such that $\mathscr{B}_{\rho, \omega}(\varphi)=\xi$ in the distributional sense, which means 
	\begin{align}\label{2.20}
		\langle\mathscr{B}_{\rho, \omega}(\varphi),v\rangle =\langle\xi,v\rangle, \ \text{ for all }\ v\in\X=\H_0^1(0,1).
	\end{align}
	Moreover, if $\xi\in\H=\L^2(0,1)$, then $\varphi\in\Z=\H^2(0,1)$, and it satisfies 
	\begin{equation}\label{2.21}
		\left\{
	\begin{aligned}
		\omega\varphi(x)-\nu\varphi_{xx}(x)+\alpha(f_{\rho}(\varphi(x)))_x-\beta \varphi(x)(1-\varphi^{\delta}(x))(\varphi^{\delta}(x)-\gamma)&=\xi(x),\\
		\varphi_x(0)-g_0(\varphi(0))=0,\  \varphi_x(1)+g_1(\varphi(1))&=0,
	\end{aligned}
	\right. 
	\end{equation}
	where the first equation in \eqref{2.21} holds for a.e. $x\in [0,1]$. Furthermore, if $\xi\in\C([0,1])$,  then $\varphi\in\C^2([0,1])$ and the first equation holds for every $x\in[0,1]$. 
\end{lemma}
For $\lambda>0$ small enough, we consider the restricted nonlinear map $(\I+\lambda\mathscr{A}) \big|_{K_{\rho}}:K_{\rho}\to\H$, where $$K_{\rho}:=\left\{v\in\D(\mathscr{A}):\|v\|_{\L^{\infty}}\leq\rho\right\}.$$ Remember that $\D(\mathscr{A})\subset\H^2(0,1)\hookrightarrow\C([0,1])$ by Sobolev's embedding.

\begin{lemma}\label{lem2.6}
	Assume that conditions of Lemma \ref{lem2.1} hold true. In addition, assume that $g_i(0)=0,$ for $i=0,1$. Then for all $\lambda\in(0,1/\omega_{\rho})$, the range $\mathrm{R}\big((\I+\lambda\mathscr{A}) \big|_{K_{\rho}}\big)\supset \overline{K_{\rho}}^{\|\cdot\|_{\H}}$, where $\overline{K_{\rho}}^{\|\cdot\|_{\H}}$ is the closure of $K_{\rho}$ in $\H=\L^2(0,1)$. 
\end{lemma}
\begin{proof}
	Let us first prove that for each $\xi\in K_{\rho},$ there is a $\varphi\in K_{\rho}$ such that $(\I+\lambda\mathscr{A})(\varphi)=\xi$. Since $\xi\in K_{\rho}\subset\X^*$, by Lemma \ref{lem2.3}, for each $\lambda\in(0,\frac{1}{\omega_{\rho}})$, there exists a $\varphi\in\X$ such that $\mathscr{B}_{\rho,\frac{1}{\lambda}}(\xi)=\frac{\xi}{\lambda}$ in the sense of distributions defined in \eqref{2.20}, which yields  $(\I+\lambda\mathscr{A}_{\rho})(\varphi)=\xi$.  In particular, since $\xi\in K_{\rho}\subset \C([0,1])$, Lemma \ref{lem2.3} implies that  there exists a $\varphi\in\C^2([0,1])$ satisfying
	\begin{equation}\label{2.22}
		\left\{
		\begin{aligned}
			\varphi(x)+\lambda\left(-\nu\varphi_{xx}(x)+\alpha(f_{\rho}(\varphi(x)))_x-\beta \varphi(x)(1-\varphi^{\delta}(x))(\varphi^{\delta}(x)-\gamma)\right)&=\xi(x),\\
			\varphi_x(0)-g_0(\varphi(0))=0,\  \varphi_x(1)+g_1(\varphi(1))&=0,
		\end{aligned}
		\right. 
	\end{equation}
	for all $x\in[0,1]$. Therefore, it is immediate that $\varphi\in\D(\mathscr{A})$. Remember the fact  $f_{\rho}(\psi)=f(\psi)$ for all $\psi\in K_{\rho}$.  Thus, in order to prove $\varphi\in K_{\rho}$  and $(\I+\lambda\mathscr{A})\varphi=\xi$, we need only to show that $-\rho\leq\phi(x)\leq \rho$  for all $x	\in[0,1]$. We exploit  the maximum principle to establish this.

Indeed, if $\varphi$ attains its positive maximum at $x_0\in(0,1)$, then it is immediate that $\varphi_{xx}(x_0)\leq 0$ and $(f_{\rho}(\varphi(x_0)))_x=f'_{\rho}(\varphi(x_0))\varphi_x(x_0)=0$. It should also be noted that
\begin{align*}
	-&\varphi(x_0)(1-\varphi^{\delta}(x_0))(\varphi^{\delta}(x_0)-\gamma)\geq 0,
\end{align*}
provided $1<\varphi^{\delta}(x_0)<\infty$ and $0<\varphi^{\delta}(x_0)<\gamma$. 
Note that $\rho\geq 1$, and if $\gamma<\varphi^{\delta}(x_0)<1$, then
$$\varphi(x)\leq\max\limits_{x\in[0,1]}\varphi(x)=\varphi(x_0)<1\leq\rho.$$ Therefore, for the cases $1<\varphi^{\delta}(x_0)<\infty$ and $0<\varphi^{\delta}(x_0)<\gamma$, from \eqref{2.22}, we infer 
\begin{align*}
	\varphi(x)&\leq\max\limits_{x\in[0,1]}\varphi(x)=\varphi(x_0)\nonumber\\&=\xi(x_0)-\lambda\left(-\nu\varphi_{xx}(x_0)+\alpha(f_{\rho}(\varphi(x_0)))_x-\beta \varphi(x_0)(1-\varphi^{\delta}(x_0))(\varphi^{\delta}(x_0)-\gamma)\right)\nonumber\\&\leq\xi(x_0)\leq\rho. 
\end{align*}

Let us now suppose that $\varphi$  attains its positive maximum at $x_0=0$. Using the boundary condition given in \eqref{2.22}, we have $\varphi_x(0)=g_0(\varphi(0))\geq 0$, since $g_0(0)=0$, $g_0(\cdot)$ is non-decreasing and $\varphi(0)>0$. In this case, $\varphi_x(0)$ cannot be positive, otherwise it would contradict the assumption that $\varphi(0)$  is a positive maximum. Therefore, $\varphi_x(0)=0$  and hence $(f_{\rho}(\varphi(0)))_x=0$. It should also be noted that $\varphi_{xx}(0)\leq 0$. If not, $\varphi_x(x)>\varphi_x(0)=0$ which implies $\varphi(x)\geq \varphi(0)$ for sufficiently small $x$.  Arguing similarly as above,  for the cases $1<\varphi^{\delta}(0)<\infty$ and $0<\varphi^{\delta}(0)<\gamma$, \eqref{2.22} provides $$\varphi(x)\leq\max\limits_{x\in[0,1]}\varphi(x)=\varphi(0)\leq\xi(0)\leq\rho.$$  The case of $\gamma<\varphi^{\delta}(0)<1$ is immediate.  In the same lines, one can show that  if $\varphi$ attains its positive maximum at $x_0=1$, then 
$$\varphi(x)\leq\max\limits_{x\in[0,1]}\varphi(x)=\varphi(1)\leq\xi(1)\leq\rho.$$ 
In a similar way, one can prove that $\varphi(x)\geq -\rho$ holds for all $x\in[0,1]$. Therefore, the proof for the case $\xi\in K_{\rho}$ is completed.

Let us now consider a general case that $\xi\in \overline{K_{\rho}}^{\|\cdot\|_{\H}}$. We can choose a sequence $\{\xi_n\}_{n\in\N}\in K_{\rho}$ such that $\xi_n\to\xi$ strongly in $\H=\L^2(0,1)$. Since $\xi_n\in K_{\rho}\subset\C([0,1])$ and $\|\xi_n\|_{\L^{\infty}}\leq\rho$, by Lemma \ref{lem2.3}  and previous arguments, one can show that the sequence $\{\varphi_n\}_{n\in\mathbb{N}}\subset\C^2([0,1])$  satisfies 
\begin{equation}\label{2.23}
		\left\{
		\begin{aligned}
			\varphi_n(x)+\lambda\left(-\nu\varphi_{n,{xx}}(x)+\alpha(f_{\rho}(\varphi_n(x)))_x-\beta \varphi_n(x)(1-\varphi_n^{\delta}(x))(\varphi_n^{\delta}(x)-\gamma)\right)&=\xi_n(x),\\
			\varphi_{n,x}(0)-g_0(\varphi_n(0))=0,\  \varphi_{n,x}(1)+g_1(\varphi_n(1))&=0,
		\end{aligned}
		\right. 
	\end{equation}
	for all $x\in[0,1]$,  and has the property that $\|\varphi_n\|_{\L^{\infty}}\leq\rho$. By the standard elliptic PDE theory, since $\C^2([0,1])\hookrightarrow\H^2(0,1)$ and the embedding $\H^2(0,1)\hookrightarrow\H^1(0,1)$ is compact, $\varphi_n\to\varphi$ strongly in $\H^1(0,1)$ and $\varphi\in\H^2(0,1)$ satisfies \eqref{2.22} for a.e. $x\in[0,1]$. The compact embedding of $\H^1(0,1)\hookrightarrow\C([0,1])$ implies  $\varphi_n\to\varphi$ pointwise and $\|\varphi\|_{\L^{\infty}}\leq\sup\limits_{n\in\N}\|\varphi_n\|_{\L^{\infty}}\leq\rho$, which proves $\varphi\in K_{\rho}$. Hence $f_{\rho}(\varphi)=f(\varphi)$. Along  with \eqref{2.22}, we conclude that $(\I+\lambda\mathscr{A}_{\rho})(\varphi)=\xi$, which completes the proof. 
\end{proof}	
An application of the Crandall-Liggett Theorem (\cite{MGTL}, \cite[Theorem 5.1, Chapter III]{JAW}) yields the following existence of a contraction $\C^0$-semigroup for the problem \eqref{1}.
\begin{theorem}\label{thm25}
	Assume that conditions of Lemmas \ref{lem2.1} and \ref{lem2.3} hold true. Then $-\mathscr{A}$ generates a nonlinear contraction $\C_0$-semigroup $\S(t)$ on $\L^{\infty}(0,1)=\bigcup\limits_{m=1}^{\infty}\overline{K_{m}}^{\|\cdot\|_{\H}},$ namely, 
	\begin{equation}
		\left\{
	\begin{aligned}
		\S(t)(v)&=\lim\limits_{n\to\infty}\left(\I+\frac{t}{n}\mathscr{A}\right)^{-1}(v)\ \text{ in }\ \H=\L^2(0,1)\\ &\qquad\text{ for all }\ v\in\L^{\infty}(0,1),\ \text{ for all }\ t>0.
	\end{aligned}
	\right. 
	\end{equation}
	Moreover, $\|\S(t)(v)\|_{\L^{\infty}}\leq\|v\|_{\L^{\infty}},$ for all $v\in\L^{\infty}(0,1)$ and $$\|\S(t)(v)-\S(t)(w)\|_{\L^{\infty}}\leq e^{\tilde{\omega}t}\|v-w\|_{\L^{\infty}},\ \text{ for all }\ v,w\in\L^{\infty}(0,1)\ \text{ and }\ t\geq 0,$$ where $\tilde{\omega}=\omega_{\max\{\|v\|_{\L^{\infty}},\|w\|_{\L^{\infty}}\}}$ is defined by \eqref{2.12}. 
\end{theorem}

\begin{proof}
	By the Crandall-Liggett Theorem, for each $\rho\geq 1$, $-\mathscr{A}\big|_{K_{\rho}}$ generates a $\C_0$-semigroup $\S_{\rho}(t)$ on $\overline{K_{\rho}}^{\|\cdot\|_{\H}}$ such that $\|\S_{\rho}(t)(v)\|_{\L^{\infty}}\leq\rho,$ for all $v\in \overline{K_{\rho}}^{\|\cdot\|_{\H}}$ and $\|\S_{\rho}(t)(v)-\S_{\rho}(t)(w)\|_{\L^{2}}\leq e^{\omega_{\rho}t}\|v-w\|_{\L^{2}},$ for all $v,w\in K_{\rho}$ and $t\geq 0$. Therefore $\S_{\rho}(t)$  is increasing in $\rho$  in the sense that $\S_{\rho_2}(t)$ is an	extension of $\S_{\rho_1}(t)$ if $\rho_2\geq\rho_1$. Therefore, the desired nonlinear contraction $\C_0$-semigroup $\S(\cdot)$ exists, which completes the proof. 
\end{proof}
Note that the definition of the $\C_0$-semigroup $\{\S(t)\}_{t\geq 0}$ is in the weak sense, that is,  $u(\cdot,t)=\S(t)(u_0(\cdot)),$ $t\geq 0$  is not necessarily differentiable with respect to time $t$, as $\mathscr{A}$ is a nonlinear operator. We should obtain stronger results so that the problem \eqref{1} has either a strong or a classical solution. Using the standard elliptic PDE arguments, one can see that $\mathscr{A}:K_m\subset\H\to\H$  is a closed nonlinear map. Furthermore, we have $\bigcup\limits_{m=1}^{\infty}K_m=\D(\mathscr{A})$. An application of \cite[Theorem 5.2, Chapter III]{MGTL} yields  the following existence and uniqueness  of strong solutions of the problem \eqref{1} along with \eqref{2p6}.
\begin{theorem}\label{thm2.6}
	Assume that the conditions of Theorem \ref{thm25} hold true. Then the nonlinear contraction semigroup $\{\S(t)\}_{t\geq 0}$ defined in  Theorem \ref{thm25} satisfies $\S(t)(\D(\mathscr{A}))\subset\D(\mathscr{A})$, and $\S(t)(u_0)$ is differentiable in $t$ for all $t>0$ provided $u_0\in\D(\mathscr{A})$. In other words, there is a \emph{unique strong solution} $u(x,t)$ of the following  GBH equation: 
	\begin{equation}
		\left\{
		\begin{aligned}
			u_t(x,t)&=\nu u_{xx}(x,t)-\alpha u^{\delta}(x,t)u_x(x,t)\\&\quad+\beta u(x,t)(1-u^{\delta}(x,t))(u^{\delta}(x,t)-\gamma),\ x\in[0,1],\ t>0,\\
				u_x(0,t)&=g_0(u(0,t)),\ t\geq 0,\\
			u_x(1,t)&=-g_1(u(1,t)), \ t\geq 0,\\
			u(x,0)&=u_0(x), \ x\in[0,1],
		\end{aligned}
		\right.
	\end{equation}
	provided $u_0\in\D(\mathscr{A})$. Moreover, $\|u(t)\|_{\L^2}\leq\|u_0\|_{\L^2}$  for all $t\geq 0$, provided $u_0\in\D(\mathscr{A})$. 
\end{theorem}
\begin{remark}
Theorem \ref{thm2.6} establishes  that $\|u(t)\|_{\H^2}$ does not blow up at any finite time $t$ provided $u_0\in\D(\mathscr{A})$.  Therefore $\|u(t)\|_{\H^1}$, $\|u(t)\|_{\L^{\infty}}$, etc. do not blow up at any finite time $t$. Nevertheless, Theorem \ref{thm2.6} does not provide explicit bounds for $\|u(t)\|_{\H^1}$, $\|u(t)\|_{\L^{\infty}}$, etc. 
\end{remark}

\subsection{Classical solutions}\label{cls}
To establish the existence and uniqueness of classical solutions of \eqref{1} subject to \eqref{2} given in the feedback form  \eqref{2p6}, one can proceed in the following way: We can consider the problem \eqref{1} as  $\mathscr{L}u\equiv u_t+\nu u_{xx}+b(u,u_x)=0$, where $b(u,u_x)=\alpha u^{\delta}u_x+\beta u(1-u^{\delta})(u^{\delta}-\gamma)$. A simple application of Young's inequality yields 
\begin {align*}
-ub(u,p)  &=-\alpha u^{\delta+1}p+\beta(\gamma+1)u^{\delta+2}-\beta\gamma u^2-\beta u^{2\delta+2}\nonumber\\&\leq\frac{\beta}{2}u^{2\delta+2}+\frac{\alpha^2}{\beta}|p|^2+\frac{\beta}{2}u^{2\delta+2}+\frac{\beta}{2}(1+\gamma)^2u^2-\beta\gamma u^2-\beta u^{2\delta+2}\nonumber\\&=\frac{\alpha^2}{\beta}|p|^2+\frac{\beta}{2}(1+\gamma^2)u^2.
\end{align*}
Therefore the existence of classical solutions follows from \cite[Theorem 7.4, Chapter V]{OAVA}. This Theorem establishes, for a more general quasi-linear parabolic boundary value problem, the existence of a unique solution in the H\"o1der space of functions $\mathcal{H}^{2+\ell,1+\ell}([0,1]\times[0,T])$ for some $\ell>0$. Since $\mathcal{H}^{2+\ell,1+\ell}([0,1]\times[0,T])\hookrightarrow\mathrm{C}^{2,1}([0,1]\times[0,T])$, we obtain	the existence of classical solutions for time intervals $[0, T]$, where $T > 0$	is arbitrarily large.  The proof in \cite{OAVA} is based on a linearization of the system and an application of the Leray-Schauder fixed point theorem.

Let us consider the following Green's function (cf. \cite{GB,JRC}):
\begin{align*}
G(x,y,t,s)&=1+2\sum_{n=1}^{\infty}\cos(n\pi x)\cos(n\pi y)e^{-\nu n^2\pi^2(t-\tau)}\nonumber\\&=\frac{1}{\sqrt{4\pi\nu(t-\tau)}}\sum_{n=-\infty}^{\infty}\left(e^{-\frac{(x-y-2n)^2}{4\nu(t-\tau)}}+e^{-\frac{(x+y-2n)^2}{4\nu(t-\tau)}}\right),
\end{align*}
corresponding to the heat operator $$\mathcal{L}u\equiv u_t-\nu u_{xx}, \ 0\leq x\leq 1, \ t>0,$$  with Neumann boundary conditions $$u_x(0,t)=g_0(u(0,t)),\ u_x(1,t)=g_1(u(1,t)), \ t>0. $$ Let us take 
\begin{equation}\label{1p3}
g(u)=-\alpha u^{\delta}u_x+\beta u(1-u^{\delta})(u^{\delta}-\gamma).
\end{equation} 
Then, the function $u(x,t) $ is a classical solution of \eqref{1}-\eqref{2} with continuous initial data $u_0(x)$ if and only if 
\begin{align*}
u(x,t)=F(u)(x,t)&\equiv\int_0^1G(x,y,t,0)u_0(y)dy+\int_0^t\int_0^1G(x,y,t,s)g(u(y,s))dyds\nonumber\\&\quad+\int_0^tG(x,1,t,s)g_1(u(1,s))ds+\int_0^tG(x,0,t,s)g_0(u(0,s))ds.
\end{align*}
The local solvability of the above equation can be established by applying the Leray-Schauder fixed point theorem to the iteration $u^{n+1}(x,t)=F(u^n)(x,t),$ $n=0,1,2,\ldots,$ with starting function $u^0(x,t)=u_0(x)$. Then by using apriori estimates (cf. \cite{AKMTM,AKMTM1,AKMTM2}, etc.), one can show that this solution can be extended as a global solution (see \cite{ABMK} for the case of viscous Burgers equation).  Smooth solutions of the system \eqref{1} along with \eqref{2p6}  should clearly be compatible with the boundary conditions at $t=0$ in some sense (\cite{OAVA}). For the definition of compatibility conditions of different order, we refer the readers to \cite{OAVA}. We summarize the discussion as the following theorem:
\begin{theorem}\label{thm2.2}
Consider the system \eqref{1} along with \eqref{2p6}. For any $T>0$, $\ell>0$  and for any $w_0\in\mathcal{H}^{2+\ell}([0,1])$ satisfying the compatibility condition 	of order $[(\ell+1)/2]$,  there exists a unique classical solution $u\in\mathcal{H}^{2+\ell,1+\ell}([0,1]\times[0,T])\hookrightarrow\mathrm{C}^{2,1}([0,1]\times[0,T])$. 
\end{theorem}
It is important to note that a crucial step in the proof is establishing \emph{uniform a priori estimates} for the problem \eqref{1}. These estimates are for the H\"older norms of solutions and hence are different from our Sobolev 	type energy estimates. The H\"older estimates establish boundedness of	solutions, while the energy estimates discussed in this paper establish stability. 

	\section{Boundary Control}\label{sec3}\setcounter{equation}{0}  
	In this section, we discuss the exponential stability for the  Burgers-Huxley  equation \eqref{1} subject to \eqref{2} for non-adaptive and adaptive control (a damped version) cases, that is, when the viscosity $\nu>\frac{\beta}{2}(1-\gamma)^2$ is known and $\nu>0$ is unknown. 
	
	The following form of Poincar\'e's inequality is used frequently in the paper (see \cite[Lemma A.1]{MK}). 
	\begin{lemma}[Poincar\'e's inequality]\label{lem3.1}
		For any $u\in\C^1([0,1])$, we have 
		\begin{align}\label{eq0}
			\int_0^1u^2(x)dx\leq 2u^2(0)+2\left(\int_0^1u_x^2(x)dx\right).
		\end{align}
	\end{lemma}

	\begin{proof}
		For any $u\in\C^1[0,1]$, we know that 
		\begin{align*}
			u(x)=u(0)+\int_0^xu_x(\zeta)d\zeta,
		\end{align*}
		so that 
		\begin{align*}
			u^2(x)&=u^2(0)+2u(0)\left(\int_0^xu_x(\zeta)d\zeta\right)+\left(\int_0^xu_x(\zeta)d\zeta\right)^2\nonumber\\&\leq 2u^2(0)+2\left(\int_0^xu_x(\zeta)d\zeta\right)^2\leq 2u^2(0)+2\left(\int_0^1u_x^2(x)dx\right).
		\end{align*}
		Ingratiating the above inequality from $0$ to $1$, we have the  Poincar\'e inequality \eqref{eq0}. 
	\end{proof}

	\subsection{Non-adaptive control}  The aim of this subsection  is to establish the $\L^2$- and $\H^1$-exponential stabilization results for the GBH  equation \eqref{1} subject to \eqref{2}, when the viscosity $\nu$ is known. 
	\begin{theorem}\label{thm2.3}
		For $$\nu>\frac{\beta}{2}(1-\gamma)^2,$$ 	the GBH equation \eqref{1} subject to \eqref{2} is exponentially stable in $\L^2(0,1)$ under the following control law: 
		\begin{equation}\label{eq6}
			\left\{
			\begin{aligned}
				w_1(t)&=(\eta_1+1)u(0,t)+\frac{\alpha}{\nu(\delta+2)}u^{\delta+1}(0,t),\ \eta_1\geq 0, \\
				w_2(t)&=-\left[\eta_2u(1,t)+\frac{\alpha}{\nu(\delta+2)}u^{\delta+1}(1,t)\right], \ \eta_2\geq 0. 
			\end{aligned}
			\right.
		\end{equation}
	\end{theorem}
	\begin{proof}
		Let $V(\cdot)$ be a  Liapunov function defined by
		\begin{align}\label{eq2}
			V(t)=\frac{1}{2}\int_0^1u^2(x,t)dx,\ t\geq 0. 
		\end{align}
		Taking the derivative with respect to $t$ and then performing an integration by parts, we deduce 
		\begin{align}\label{eq1}
			\dot{V}(t)&=\int_0^1u(x,t)u_t(x,t)dx\nonumber\\&=\int_0^1u(x,t)[\nu u_{xx}(x,t)-\alpha u^{\delta}(x,t)u_x(x,t)+\beta u(x,t)(1-u^{\delta}(x,t))(u^{\delta}(x,t)-\gamma)]dx\nonumber\\&=-\nu\int_0^1u_x^2(x,t)dx+\nu u(1,t)u_x(1,t)-\nu u(0,t)u_x(0,t)-\frac{\alpha}{\delta+2}\int_0^1(u^{\delta+2}(x,t))_xdx\nonumber\\&\quad+\beta(1+\gamma)\int_0^1u^{\delta+2}(x,t)dx-\beta\gamma\int_0^1u^2(x,t)dx-\beta\int_0^1u^{2\delta+2}(x,t)dx \nonumber\\&= -\nu\int_0^1u_x^2(x,t)dx+\nu u(1,t)w_2(t)-\nu u(0,t)w_1(t)-\frac{\alpha}{\delta+2}u^{\delta+2}(1,t)\nonumber\\&\quad+\frac{\alpha}{\delta+2}u^{\delta+2}(0,t)+\beta(1+\gamma)(u(t),u^{\delta+1}(t))-\beta\gamma\|u(t)\|_{\L^2}^2-\beta\|u(t)\|_{\L^{2(\delta+1)}}^{2(\delta+1)}\nonumber\\&\leq -\frac{\nu}{2}\int_0^1u^2(x,t)dx+\nu u^2(0,t)+\nu u(1,t)w_2(t)-\nu u(0,t)w_1(t)-\frac{\alpha}{\delta+2}u^{\delta+2}(1,t)\nonumber\\&\quad+\frac{\alpha}{\delta+2}u^{\delta+2}(0,t)+\left(\frac{\beta(1+\gamma)^2}{4\theta}-\beta\gamma\right)\|u(t)\|_{\L^2}^2-(1-\theta)\beta\|u(t)\|_{\L^{2(\delta+1)}}^{2(\delta+1)}\nonumber\\&\leq-\left(\frac{\nu}{2}+\beta\gamma-\frac{\beta(1+\gamma)^2}{4\theta}\right)V(t)-\nu u(0,t)\left\{w_1(t)-u(0,t)-\frac{\alpha}{\nu(\delta+2)}u^{\delta+1}(0,t)\right\}\nonumber\\&\quad+\nu u(1,t)\left\{w_2(t)-\frac{\alpha}{\nu(\delta+2)}u^{\delta+1}(1,t)\right\},
		\end{align}
		for $0<\theta\leq1$, where we have used H\"older's, Young's and Poincar\'e's inequalities (Lemma \ref{lem3.1}). If we apply the  control law given in \eqref{eq6},  then \eqref{eq1} reduces to 
		\begin{align}\label{35}
			\dot{V}(t)&\leq -\left(\frac{\nu}{2}+\beta\gamma-\frac{\beta(1+\gamma)^2}{4\theta}\right)V(t)-\nu(\eta_1u^2(0,t)+\eta_2u^2(1,t))\leq -\zeta V(t),
		\end{align}
		where 
		\begin{align}\label{eq3}
			\zeta=\frac{\nu}{2}+\beta\gamma-\frac{\beta(1+\gamma)^2}{4\theta}>0,
		\end{align} 
		for $0<\theta\leq 1$ and $\nu>\frac{\beta}{2}(1-\gamma)^2$. Thus we deduce  $V(t)\leq V(0)e^{-\zeta t},$ for all $t\geq 0$ and since $\zeta>0$,  $V(t)$ converges to zero exponentially as $t\to\infty$. 
	\end{proof}

	\begin{remark}
		One can use the control 
			$w_2(t)=-\eta_2u(1,t)+\frac{\alpha}{\nu(\delta+2)}u^{\delta+1}(1,t), \ \eta_2\geq 0$ also in \eqref{eq6}  to stabilize the problem \eqref{1}. But it is not in the form of the boundary data that we have considered in \eqref{2p6}. 
	\end{remark}

	We need  the following lemma to obtain asymptotic stability results for the forced GBH equation as well as non-adaptive control problem: 
	\begin{lemma}[Lemma 2, \cite{NS1}]\label{lem1}
		Let $\zeta>0$. For any $u\in \L^{p}(0,\infty)$, $p\in[1,\infty)$, we have 
		\begin{align}
			\int_0^te^{-\zeta(t-s)}|u^{p}(k,s)|ds\to 0\ \text{ as }\ t\to\infty,
		\end{align}
		where $k=0,1$. 
	\end{lemma}
	
	Let us now consider the following  forced GBH equation:
	\begin{align}\label{111}
		u_t(x,t)-\nu u_{xx}(x,t)+\alpha u^{\delta}(x,t)u_x(x,t)-\beta u(x,t)(1-u^{\delta}(x,t))(u^{\delta}(x,t)-\gamma)=f(x,t), 
	\end{align}
for $x\in[0,1]$, $t\geq 1$. The well-posedness of the forced GBH boundary value problem for $f\in\L^{\infty}((0,T)\times(0,1))$ can be obtained by using the similar arguments as in \cite[Theorem 3.1]{KIYY}. 	Moreover, we have the following asymptotic stability result:
\begin{theorem}\label{rem2.4}
	For $\nu>\frac{\beta}{2}(1-\gamma)^2,$  and $f\in \L^{2}(0,\infty;\L^2(0,1))$ or $\lim\limits_{t\to\infty}\sup\limits_{0\leq s\leq t}\|f(s)\|_{\L^2}=0$, 	the forced GBH equation \eqref{111} subject to \eqref{2} is asymptotically  stable in $\L^2(0,1)$.
\end{theorem}

	\begin{proof}
A similar calculation as we have performed  in \eqref{35} yields 
		\begin{align*}
			\dot{V}(t)\leq -\left(\frac{(1-\hat\theta)\nu}{2}+\beta\gamma-\frac{\beta(1+\gamma)^2}{4\theta}\right)V(t)+\frac{1}{2\hat\theta\nu}\|f(t)\|_{\L^2}^2,
		\end{align*}
		for $0<\theta,\hat\theta<1$. Taking $\hat{\zeta}=\left(\frac{(1-\hat\theta)\nu}{2}+\beta\gamma-\frac{\beta(1+\gamma)^2}{4\hat\theta}\right)>0$, and applying  of the variation of constants formula, we arrive at 
		\begin{align}
			V(t)\leq V(0)e^{-\hat\zeta t}+\frac{1}{2\hat\theta\nu}\int_0^te^{-\hat\zeta (t-s)}\|f(s)\|_{\L^2}^2ds.
		\end{align}
		For $\nu>\frac{\beta}{2}(1-\gamma)^2$ and $f\in \L^{2}(0,\infty;\L^2(0,1))$ or $\lim\limits_{t\to\infty}\sup\limits_{0\leq s\leq t}\|f(s)\|_{\L^2}=0$, using Lemma \ref{lem1}, one can obtain the asymptotic stability. 
	\end{proof}

	\begin{remark}\label{rem2.5}
		For Dirichlet's boundary control:
		\begin{align*}
			u(0,t)=\tilde w_1(t)\ \text{ and }\ u(1,t)=\tilde w_2(t),
		\end{align*}
		one can design a control law for the exponential stability in $\L^2(0,1)$ as a solution for the following system: 
		\begin{equation*}
			\left\{
			\begin{aligned}
				&\tilde w_1(t)=\left\{\frac{\nu(\delta+2)}{\alpha}\left(u_x(0,t)-(1+\tilde\eta_1)\tilde w_1(t)\right)\right\}^{\frac{1}{\delta+1}},\ \tilde\eta_1\geq 0, \\
				&\tilde w_2(t)=\left\{\frac{\nu(\delta+2)}{\alpha}\left[\tilde\eta_2\tilde w_2(t)+u_x(1,t)\right]\right\}^{\frac{1}{\delta+1}}, \tilde\eta_2\geq 0. 
			\end{aligned}\right.
		\end{equation*}
	\end{remark}

	Let us now state and prove a general version of Theorem \ref{thm2.3}. 
	\begin{theorem}\label{thm3.6}
Global exponential stability in the $\L^q(0,1)$ sense holds for the problem \eqref{1}-\eqref{2}, provided $\nu>\frac{\beta p}{2(2p-1)}(1-\gamma)^2$,  that is, for any $q\in[2,\infty)$, there exists $\bar{\zeta}(q)$ such that 
		\begin{align}\label{316}
			\|u(t)\|_{\L^q}\leq \|u_0\|_{\L^q}e^{-\bar{\zeta} t}, \ \text{ for all }\ t\geq 0. 
		\end{align}
	\end{theorem}
	\begin{proof}
		If one considers a Liapunov function 
	\begin{align}\label{2p9}
		\tilde{V}(t)=\frac{1}{2p}\int_0^1u^{2p}(x,t)dx=\frac{1}{2p}\|u(t)\|_{\L^{2p}}^{2p},
	\end{align}
	then, we have  for all $t\geq 0$
	\begin{align}\label{312}
		\dot{\tilde{V}}(t)&=\int_0^1u^{2p-1}(x,t)u_t(x,t)dx\nonumber\\&=\int_0^1u^{2p-1}(x,t)[\nu u_{xx}(x,t)-\alpha u^{\delta}(x,t)u_x(x,t)+\beta u(x,t)(1-u^{\delta}(x,t))(u^{\delta}(x,t)-\gamma)]dx\nonumber\\&=-(2p-1)\nu\int_0^1u^{2p-2}(x,t)u_x^2(x,t)dx+\nu u^{2p-1}(1,t)u_x(1,t)-\nu u^{2p-1}(0,t)u_x(0,t)\nonumber\\&\quad-\frac{\alpha}{2p+\delta}u^{2p+\delta}(1,t)+\frac{\alpha}{2p+\delta}u^{2p+\delta}(0,t)+\beta(1+\gamma)(u^{2p-1}(t),u^{\delta+1}(t))\nonumber\\&\quad-\beta\gamma\|u(t)\|_{\L^{2p}}^{2p}-\beta\|u(t)\|_{\L^{2(\delta+p)}}^{2(\delta+p)}\nonumber\\&\leq -\left(\frac{\nu (2p-1)}{2p}+\beta\gamma-\frac{\beta(1+\gamma)^2}{4}\right)\int_0^1u^{2p}(x,t)dx\nonumber\\&\quad-\nu u^{2p-1}(0,t) \left( w_1(t)-\frac{(2p-1)}{p}u(0,t)-\frac{\alpha}{\nu(2p+\delta)}u^{\delta+1}(0,t)\right)\nonumber\\&\quad+\nu u^{2p-1}(1,t)\left(w_2(t)-\frac{\alpha}{\nu(2p+\delta)}u^{\delta+1}(1,t)\right),
	\end{align}
	where we have used Poincar\'e's inequality (Lemma \ref{lem3.1}). 	By choosing the boundary control law  as
	\begin{equation}\label{2p10}
		\left\{
		\begin{aligned}
			w_1(t)&=\left(\hat\eta_1+\frac{(2p-1)}{p}\right)u(0,t)+\frac{\alpha}{\nu(2p+\delta)}u^{\delta+1}(0,t),\ \hat\eta_1\geq 0, \\
			w_2(t)&=-\left[\hat\eta_2u(1,t)+\frac{\alpha}{\nu(2p+\delta)}u^{\delta+1}(1,t)\right], \ \hat\eta_2\geq 0,
		\end{aligned}
		\right.
	\end{equation}
	in \eqref{2p9}, we find 
	\begin{align}
		\frac{d}{d t}\|u(t)\|_{\L^{2p}}^{2p}\leq -2p\left(\frac{\nu (2p-1)}{2p}+\beta\gamma-\frac{\beta(1+\gamma)^2}{4}\right)\|u(t)\|_{\L^{2p}}^{2p},
	\end{align}
	for all $t\geq 0$.	Therefore, for $\nu>\frac{\beta p}{2(2p-1)}(1-\gamma)^2$, one  obtains \begin{align}\label{314}\|u(t)\|_{\L^{2p}}\leq\|u_0\|_{\L^{2p}}e^{-\bar\zeta t},\end{align}  where $\bar{\zeta}=\left(\frac{\nu (2p-1)}{2p}-\frac{\beta(1-\gamma)^2}{4}\right)>0.$ The estimate \eqref{316} follows from \eqref{314}.  Thus, the solution $u\equiv 0$ is globally exponentially stable in an $\L^q$-sense, where $q\in[2,\infty)$.  
	\end{proof}

		\begin{remark}
			1. Note that the control given in \eqref{2p10} for $2p=q$  exponentially stabilizes any $\L^r$-norm for any $2\leq r\leq q$.

		2. Letting $p\to\infty$ in \eqref{314}, one can deduce that 
		\begin{align}
			\esssup_{x\in[0,1]}|u(x,t)|\leq\esssup_{x\in[0,1]}|u(x,0)|, \ \text{ for all }\ t\geq 0. 
		\end{align}
		But the the above result is not particularly useful for two reasons:
		\begin{enumerate}
			\item [(1)] The above estimate does not guarantee convergence to zero, that is, it
			guarantees stability but not asymptotic stability.
			\item [(2)] With $\esssup$, we cannot guarantee boundedness for all (but only for almost all) $x\in[0,1]$ without further effort to demonstrate continuity.
		\end{enumerate} 
	\end{remark}

		\begin{theorem}\label{thm3.8}
			Global exponential stability in the $\H^1$-norm 	sense holds for the problem \eqref{1}-\eqref{2}, that is,  for $\nu>\frac{2p}{2p-1}\beta(1+\gamma^2)$, 
				\begin{align}\label{3p17}
			\|u(t)\|_{\H^1}^2=	\|u(t)\|_{\L^2}^2+	\|u_x(t)\|_{\L^2}^2\leq C\left(\varrho,\hat\eta_1,\hat\eta_2,\nu,\alpha,\beta,\gamma,\delta,\|(u_0)_x\|_{\L^2}\right)e^{-\frac{\varrho}{2} t},
				\end{align}
				for all $t\geq 0$ and $u_0\in\H^1(0,1)$, where 
				\begin{align}\label{3p18}
					\varrho= 2p\left(\frac{\nu(2p-1)}{4p}-\frac{\beta(1+\gamma^2)}{2}\right)>0\ \text{ and }\ p=\delta+1.
				\end{align}
		\end{theorem}
		\begin{proof}
		Let us now take the $\L^2$-inner product of \eqref{1} with $-u_{xx}$ to infer 
		\begin{align}\label{317}
			-&\int_0^1u_t(x,t)u_{xx}(x,t)dx+\nu\int_0^1u_{xx}^2(x,t)dx\nonumber\\&=\alpha \int_0^1u^{\delta}(x,t)u_x(x,t)u_{xx}(x,t)dx-\beta(1+\gamma)\int_0^1u^{\delta+1}(x,t)u_{xx}(x,t)dx\nonumber\\&\quad+\beta\gamma\int_0^1u(x,t)u_{xx}(x,t)dx+\beta\int_0^1u^{2\delta+1}(x,t)u_{xx}(x,t)dx,
		\end{align}
		for a.e. $t\geq 0$. 
		 The first term in the left hand side of \eqref{317} can be estimated as follows: 
		\begin{align}\label{318}
			-\int_0^1u_t(x,t)u_{xx}(x,t)dx&=-u_t(x,t)u_x(x,t)\big|_0^1+\int_0^1u_{xt}(x,t)u_x(x,t)dx\nonumber\\&=\frac{1}{2}\frac{d}{dt}\|u_x(t)\|_{\L^2}^2+u_t(1,t)\left[\hat\eta_2u(1,t)+\frac{\alpha}{\nu(2p+\delta)}u^{\delta+1}(1,t)\right]\nonumber\\&\quad+u_t(0,t)\left[\left(\hat\eta_1+\frac{(2p-1)}{p}\right)u(0,t)+\frac{\alpha}{\nu(2p+\delta)}u^{\delta+1}(0,t)\right]\nonumber\\&=\frac{1}{2}\frac{d}{dt}\bigg[\|u_x(t)\|_{\L^2}^2+\frac{\hat{\eta}_2}{2}u^2(1,t)+\frac{\alpha}{\nu(2p+\delta)(\delta+2)}u^{\delta+2}(1,t)\nonumber\\&\quad+\left(\frac{\hat\eta_1}{2}+\frac{(2p-1)}{2p}\right)u^2(0,t)+\frac{\alpha}{\nu(2p+\delta)(\delta+2)}u^{\delta+2}(0,t)\bigg].
		\end{align}
		Using H\"older's and Young's inequalities, we estimate the first and second terms from the right hand side of \eqref{317} as 
		\begin{align}
			\alpha \int_0^1u^{\delta}(x,t)u_x(x,t)u_{xx}(x,t)dx&\leq\|u_{xx}(t)\|_{\L^2}\|u^{\delta}(t)u_x(t)\|_{\L^2}\nonumber\\&\leq\frac{\nu}{4}\|u_{xx}(t)\|_{\L^2}^2+\frac{\alpha^2}{\nu}\|u^{\delta}(t)u_x(t)\|_{\L^2}^2\nonumber\\\beta(1+\gamma)\int_0^1u^{\delta+1}(x,t)u_{xx}(x,t)dx&\leq\beta(1+\gamma)\|u_{xx}(t)\|_{\L^2}\|u(t)\|_{\L^{2(\delta+1)}}^{\delta+1}\nonumber\\&\leq\frac{\nu}{4}\|u_{xx}(t)\|_{\L^2}^2+\frac{\beta^2(1+\gamma)^2}{\nu}\|u(t)\|_{\L^{2(\delta+1)}}^{2(\delta+1)}. 
		\end{align}
		We estimate the final two terms from the right hand side of \eqref{317} by performing an integration by parts as 
	\begin{align}
	\int_0^1u(x,t)u_{xx}(x,t)dx&=\left[u(x,t)u_{x}(x,t)\big|_0^1-\int_0^1u_{x}^2(x,t)dx\right]\nonumber\\&=-\bigg\{u(1,t)\left[\hat\eta_2u(1,t)+\frac{\alpha}{\nu(2p+\delta)}u^{\delta+1}(1,t)\right]\nonumber\\&\quad+u(0,t)\left[\left(\hat\eta_1+\frac{(2p-1)}{p}\right)u(0,t)+\frac{\alpha}{\nu(2p+\delta)}u^{\delta+1}(0,t)\right]\nonumber\\&\quad+\int_0^1u_{x}^2(x,t)dx\bigg\}\leq 0,\\
		\int_0^1u^{2\delta+1}(x,t)u_{xx}(x,t)dx&=\left[u2\delta+1(x,t)u_{x}(x,t)\big|_0^1-(2\delta+1)\int_0^1u^{2\delta}(x,t)u_{x}^2(x,t)dx\right]\nonumber\\&=-\bigg\{u^{2\delta+1}(1,t)\left[\hat\eta_2u(1,t)+\frac{\alpha}{\nu(2p+\delta)}u^{\delta+1}(1,t)\right]\nonumber\\&\quad+u^{2\delta+1}(0,t)\left[\left(\hat\eta_1+\frac{(2p-1)}{p}\right)u(0,t)+\frac{\alpha}{\nu(2p+\delta)}u^{\delta+1}(0,t)\right]\nonumber\\&\quad+\int_0^1u^{2\delta}(x,t)u_{x}^2(x,t)dx\bigg\}\leq 0.\label{321}
	\end{align}
	Combining \eqref{318}-\eqref{321} and then substituting it in \eqref{317}, we find 
	\begin{align}\label{322}
		\frac{1}{2}\frac{d}{dt}\Psi(t)+\frac{\nu}{2}\|u_{xx}(t)\|_{\L^2}^2+\beta\|u^{\delta}(t)u_x(t)\|_{\L^2}^2\leq \frac{\alpha^2}{\nu}\|u^{\delta}(t)u_x(t)\|_{\L^2}^2 +\frac{\beta^2(1+\gamma)^2}{\nu}\|u(t)\|_{\L^{2(\delta+1)}}^{2(\delta+1)},
	\end{align}
for a.e. $t\geq 0$,	where 
	\begin{align}\label{3p23}
		\Psi(t)&=\|u_x(t)\|_{\L^2}^2+\frac{\hat{\eta}_2}{2}u^2(1,t)+\frac{\alpha}{\nu(2p+\delta)(\delta+2)}u^{\delta+2}(1,t)\nonumber\\&\quad+\left(\frac{\hat\eta_1}{2}+\frac{(2p-1)}{2p}\right)u^2(0,t)+\frac{\alpha}{\nu(2p+\delta)(\delta+2)}u^{\delta+2}(0,t).
	\end{align}
	We  multiply \eqref{322} by $e^{\frac{\varrho}{2} t}$ to deduce 
	\begin{align}\label{323}
	&	\frac{d}{dt}[e^{\frac{\varrho}{2} t}\Psi(t)]+\nu e^{\frac{\varrho}{2} t}\|u_{xx}(t)\|_{\L^2}^2+2\beta e^{\frac{\varrho}{2} t}\|u^{\delta}(t)u_x(t)\|_{\L^2}^2\nonumber\\&\leq \frac{\varrho}{2} e^{\frac{\varrho}{2} t}\Psi(t)+\frac{2\alpha^2}{\nu}e^{\frac{\varrho}{2} t}\|u^{\delta}(t)u_x(t)\|_{\L^2}^2 +\frac{2\beta^2(1+\gamma)^2}{\nu}e^{\frac{\varrho}{2} t}\|u(t)\|_{\L^{2(\delta+1)}}^{2(\delta+1)},
	\end{align}
where $\varrho$ is defined in \eqref{3p18}. 	Since, we have the embedding $\W^{1,p}(0,1)\hookrightarrow\C([0,1])$ for any $1\leq p<\infty$ (Remark \ref{rem2.2}), one can estimate 
	\begin{align*}
		u^2(1,t)&\leq\max\limits_{x\in[0,1]}|u^2(x,t)|\leq C_a(\|u(t)\|_{\L^2}^2+\|u_x(t)\|_{\L^2}^2),\\
		u^{\delta+2}(1,t)&\leq\max\limits_{x\in[0,1]}|u^{\delta+2}(x,t)|\leq C_b\|u^{\delta+2}(t)\|_{\W^{1,1}}=C_b\|u^{\delta+2}(t)\|_{\L^1}+C_b\|(u^{\delta+2}(t))_x\|_{\L^1}\nonumber\\&\leq C_b\|u(t)\|_{\L^{\delta+2}}^{\delta+2}+C_b(\delta+2)\|u^{\delta+1}(t)u_x(t)\|_{\L^1}\nonumber\\&\leq C_b\|u(t)\|_{\L^{\delta+2}}^{\delta+2}+C_b(\delta+2)\|u^{\delta}(t)u_x(t)\|_{\L^{2}}\|u_x(t)\|_{\L^2}\nonumber\\&\leq\beta\|u^{\delta}(t)u_x(t)\|_{\L^{2}}^2+C_b\|u(t)\|_{\L^{\delta+2}}^{\delta+2}+\frac{C_b^2(\delta+2)^2}{4\beta}\|u_x(t)\|_{\L^2}^2,
	\end{align*}
	for some constants $C_a$ and $C_b$. Similar calculations hold for the quantities $u^2(0,t)$ and  $u^{\delta+2}(0,t)$. Therefore, \eqref{323} implies 
	\begin{align}
		&	\frac{d}{dt}[e^{\frac{\varrho}{2} t}\Psi(t)]+\nu e^{\frac{\varrho}{2} t}\|u_{xx}(t)\|_{\L^2}^2+\beta e^{\frac{\varrho}{2} t}\|u^{\delta}(t)u_x(t)\|_{\L^2}^2\nonumber\\&\leq C_{\varrho,\hat\eta_1,\hat\eta_2,\nu,\alpha,\beta,\gamma,\delta}e^{\frac{\varrho}{2} t}\left[\|u(t)\|_{\L^{\delta+2}}^{\delta+2}+\|u_x(t)\|_{\L^2}^2+\|u^{\delta}(t)u_x(t)\|_{\L^2}^2+\|u(t)\|_{\L^{2(\delta+1)}}^{2(\delta+1)}\right],
	\end{align}
	for a.e. $t\geq 0$. Integrating the above inequality from $0$ to $t$, we find 
	\begin{align}\label{325}
		e^{\frac{\varrho}{2} t}\Psi(t)&\leq \Psi(0)+C_{\varrho,\hat\eta_1,\hat\eta_2,\nu,\alpha,\beta,\gamma,\delta}\nonumber\\&\quad\times\int_0^te^{\frac{\varrho}{2} s}\left[\|u(s)\|_{\L^{\delta+2}}^{\delta+2}+\|u_x(s)\|_{\L^2}^2+\|u^{\delta}(s)u_x(s)\|_{\L^2}^2+\|u(s)\|_{\L^{2(\delta+1)}}^{2(\delta+1)}\right]d s,
	\end{align}
	for all $t\geq 0.$ 
	
	A calculation similar to \eqref{312} yields 
	\begin{align}
	&	\frac{1}{2p}\frac{d}{d t}\|u(t)\|_{\L^{2p}}^{2p}+\frac{(2p-1)\nu}{2}\|u^{p-1}(t)u_x^2(t)\|_{\L^2}^2+\frac{\beta}{2}\|u(t)\|_{\L^{2(\delta+p)}}^{2(\delta+p)}\nonumber\\&\leq -\left(\frac{\nu(2p-1)}{4p}+\beta\gamma-\frac{\beta(1+\gamma)^2}{2}\right)\|u(t)\|_{\L^{2p}}^{2p},
	\end{align}
	for a.e. $t\in[0,T]$. Therefore, for $\nu>\frac{2p}{2p-1}\beta(1+\gamma^2)$ and $\varrho=2p\left(\frac{\nu(2p-1)}{4p}-\frac{\beta(1+\gamma^2)}{2}\right)>0$, the above estimate immediately gives 
	\begin{align}\label{326}
		e^{\varrho t}\|u(t)\|_{\L^{2p}}^{2p}+p(2p-1)\nu\int_0^te^{\varrho s}\|u^{p-1}(s)u_x^2(s)\|_{\L^2}^2ds+p\beta\int_0^te^{\varrho s}\|u(s)\|_{\L^{2(\delta+p)}}^{2(\delta+p)}ds\leq \|u_0\|_{\L^{2p}}^{2p},
	\end{align}
	for all $t\geq 0.$ Choosing $p=\delta+1,1$, we immediately have 
	\begin{align*}
		&\|u(t)\|_{\L^{\delta+2}}\leq\|u(t)\|_{\L^{2(\delta+1)}}\leq e^{-\frac{\varrho t}{2(\delta+1)}}\|u_0\|_{\L^{2(\delta+1)}},\\
	&	\int_0^te^{\varrho s}\|u^{\delta}(s)u_x^2(s)\|_{\L^2}^2ds\leq C_{\nu,\delta}\|u_0\|_{\L^{2(\delta+1)}}^{2(\delta+1)},\\
	&	\int_0^te^{\varrho s}\|u(s)\|_{\L^{2(\delta+p)}}^{2(\delta+p)}ds\leq C_{\beta,\delta}\|u_0\|_{\L^{2(\delta+1)}}^{2(\delta+1)},\\
	&	\int_0^te^{\varrho s}\|u_x^2(s)\|_{\L^2}^2ds+\int_0^te^{\varrho s}\|u(s)\|_{\L^{2(\delta+1)}}^{2(\delta+1)}ds\leq C_{\nu,\beta,\delta}\|u_0\|_{\L^{2}}^{2}.
		\end{align*}
	Substituting the above inequalities in \eqref{325}, we deduce 
	\begin{align}
			\Psi(t)&\leq\left\{ \Psi(0)+C_{\varrho,\hat\eta_1,\hat\eta_2,\nu,\alpha,\beta,\gamma,\delta}\left(\|u_0\|_{\L^2}^2+\|u_0\|_{\L^{2(\delta+1)}}^{\delta+2}+\|u_0\|_{\L^{2(\delta+1)}}^{2(\delta+1)}\right)\right\}e^{-\frac{\varrho}{2} t},
	\end{align}
	for all $t\geq 0$. From the definition of $\Psi$ given in \eqref{3p23} and Sobolev's inequality, one can immediately deduce \eqref{3p17}. 
	\end{proof}
	
	\begin{remark}
		As we are not using any Gronwall's inequality to obtain \eqref{325}, we are able to establish global exponential stability in $\H^1$-norm. The author in \cite{ABMK} obtained a semi-global exponential stability only in $\H^1$-norm for the viscous Burgers equation. 
	\end{remark}
	Therefore an application of Agmon's inequality (Lemma \ref{lem21}) yields  the following result on the pointwise exponential convergence: 
	\begin{theorem}\label{thm3.10}
		Consider the problem \eqref{1} with boundary control \eqref{2p10} for $p=\delta+1$. If $w_0\in\H^1(0,1)$ and $\nu>\frac{2p}{2p-1}\beta(1+\gamma^2)$, 
		\begin{align}
		\max_{x\in[0,1]}|u(x,t)|\leq C\left(\varrho,\hat\eta_1,\hat\eta_2,\nu,\alpha,\beta,\gamma,\delta,\|(u_0)_x\|_{\L^2}\right)e^{-\frac{\varrho}{2} t}, 
	\end{align}
for all $t\geq 0$,	where $\varrho$ is defined in \eqref{3p18}. 
	\end{theorem}

	\subsection{Adaptive control} In this subsection, we consider the adaptive control problem of the following  damped gneralized Burgers-Huxley equation subject to \eqref{2} (that is, when $\nu>0$ is unknown):
	\begin{align}\label{1d}
		u_t(x,t)=\nu u_{xx}(x,t)-\kappa u(x,t)-\alpha u^{\delta}(x,t)u_x(x,t)+\beta u(x,t)(1-u^{\delta}(x,t))(u^{\delta}(x,t)-\gamma), 
	\end{align}
	where $0<\kappa\leq \frac{\beta}{4}(1-\gamma)^2$ is known. 
	\begin{theorem}\label{thm2.5}
		The solution $u(\cdot)$ of the closed-loop system of the damped GBH equation \eqref{1d} subject to \eqref{2} with unknown viscosity $\nu>0$, is asymptotically stable under the following control law:
		\begin{equation}\label{29}
			\left\{
			\begin{aligned}
				w_1(t)&=\eta_1(t)u^{\delta+1}(0,t)+u(0,t),\\
				w_2(t)&=\eta_2(t)u^{\delta+1}(1,t),
			\end{aligned}
			\right.
		\end{equation}
		where $\eta_k(t)$, $k=1,2,$ are bounded for any $t\geq 0$ with: 
		\begin{equation}\label{eq4}
			\left\{
			\begin{aligned}
				\dot{\eta}_1(t)&=r_1u^{\delta+2}(0,t),\ r_1>0,\\
				\dot{\eta}_2(t)&=-r_2u^{\delta+2}(1,t),\ r_2>0. 
			\end{aligned}
			\right.
		\end{equation}
	\end{theorem}
	\begin{proof}
		Let $\mathcal{E}(t)$ be a non-negative energy function defined by 
		\begin{align}
			\mathcal{E}(t)=V(t)+\left(\frac{1}{2\nu r_1}\right)\left(\nu\eta_1(t)-\frac{\alpha}{\delta+2}\right)^2+\left(\frac{1}{2\nu r_2}\right)\left(\nu\eta_2(t)-\frac{\alpha}{\delta+2}\right)^2,
		\end{align}
		where $V(t)$ is defined in \eqref{eq2}. Taking the time derivative of the energy function $\mathcal{E}(t)$, we get 
		\begin{align}\label{eq5}
			\dot{\mathcal{E}}(t)&=\dot{V}(t)+\frac{1}{r_1}\left(\nu\eta_1(t)-\frac{\alpha}{\delta+2}\right)\dot{\eta}_1(t)+\frac{1}{r_2}\left(\nu\eta_2(t)-\frac{\alpha}{\delta+2}\right)\dot{\eta}_2(t)\nonumber\\&\leq -\zeta V(t)-\left(\nu\eta_1(t)-\frac{\alpha}{\delta+2}\right)u^{\delta+2}(0,t)+\left(\nu\eta_2(t)-\frac{\alpha}{\delta+2}\right)u^{\delta+2}(1,t)\nonumber\\&\quad+\frac{1}{r_1}\left(\nu\eta_1(t)-\frac{\alpha}{\delta+2}\right)\dot{\eta}_1(t)+\frac{1}{r_2}\left(\nu\eta_2(t)-\frac{\alpha}{\delta+2}\right)\dot{\eta}_2(t),
		\end{align}
		where we have used \eqref{eq1} with $\left(\frac{\nu}{2}+\beta\gamma-\frac{\beta(1+\gamma)^2}{4\theta}\right)$ replaced by $\left(\frac{\nu}{2}+\kappa+\beta\gamma-\frac{\beta(1+\gamma)^2}{4\theta}\right)$ and $
			\zeta=\frac{\nu}{2}+\kappa+\beta\gamma-\frac{\beta(1+\gamma)^2}{4\theta}>0,$ since $\kappa\leq \frac{\beta}{4}(1-\gamma)^2$.  Substituting \eqref{eq4} in \eqref{eq5}, we obtain 
		\begin{align}
			\dot{\mathcal{E}}(t)\leq -\zeta V(t)\leq 0, 
		\end{align}
		since $\zeta>0$. Therefore, we have $\mathcal{E}(t)\leq \mathcal{E}(0)$, and  it is immediate that $\eta_1(t)$ and $\eta_2(t)$ are bounded functions for any $t>0$. This also implies $u(0,t),u(1,t)\in  \L^{\delta+2}(0,\infty)$. 
		
		Let now show the asymptotic stability result. From \eqref{eq5}, we have 
		\begin{align}
			\dot{V}(t)\leq -\zeta V(t)-\nu\left(\eta_1(t)-\frac{\alpha}{\nu(\delta+2)}\right)u^{\delta+2}(0,t)+\nu\left(\eta_2(t)-\frac{\alpha}{\nu(\delta+2)}\right)u^{\delta+2}(1,t).
		\end{align}
		An application of variation of constants formula yields 
		\begin{align}
			V(t)&\leq V(0)e^{-\zeta t}\nonumber\\&\quad+\nu\int_0^t\left\{\left(\frac{\alpha}{\nu(\delta+2)}-\eta_1(s)\right)u^{\delta+2}(0,s)+\left(\eta_2(s)-\frac{\alpha}{\nu(\delta+2)}\right)u^{\delta+2}(1,s)\right\}e^{-\zeta(t-s)}ds\nonumber\\&\leq V(0)e^{-\zeta t}+\nu\xi\left(\int_0^te^{-\zeta(t-s)}|u^{\delta+2}(0,s)|ds+\int_0^te^{-\zeta(t-s)}|u^{\delta+2}(1,s)|ds\right),
		\end{align}
		where $$\xi=\max\left\{\sup_{s\in[0,\infty)}\left|\frac{\alpha}{\nu(\delta+2)}-\eta_1(s)\right|,\sup_{s\in[0,\infty)}\left|\eta_2(s)-\frac{\alpha}{\nu(\delta+2)}\right|\right\}.$$ 
		Exploiting Lemma \ref{lem1}, one can easily obtain the required result. 
	\end{proof}
	
	\begin{remark}
		Following the methodology adopted in \cite{KIYY}, we can consider a general conservation form of \eqref{1} and obtain boundary feedback control results, since we have already established the existence and uniqueness of strong solutions of the system \eqref{1} with $\frac{1}{\delta+1}v^{\delta+1}$ replaced by $f(v)$ for $f\in\C^1(\R)$ (Theorem \ref{thm2.6}).
	\end{remark}

	\section{Numerical results}\label{sec4}\setcounter{equation}{0}  
	
In our numerical study, we initially discretize the given equation with the help of Chebyshev nodes within the interval \([0,1]\) using a transformation from the standard Chebyshev nodes on \([-1, 1]\). For our analysis, we select $50$ nodes. A MATLAB program employing the Chebyshev collocation method for spatial discretization and the backward Euler method for time discretization has been developed. It solves both the uncontrolled and controlled versions of the Generalized Burgers-Huxley (GBH) equation \eqref{1}, as well as the adaptive control of the unforced damped GBH equation \eqref{1d}, subject to conditions outlined in equations \eqref{29} and \eqref{eq4}.
\begin{figure}[h]
	\centerline{\includegraphics[width=16cm]{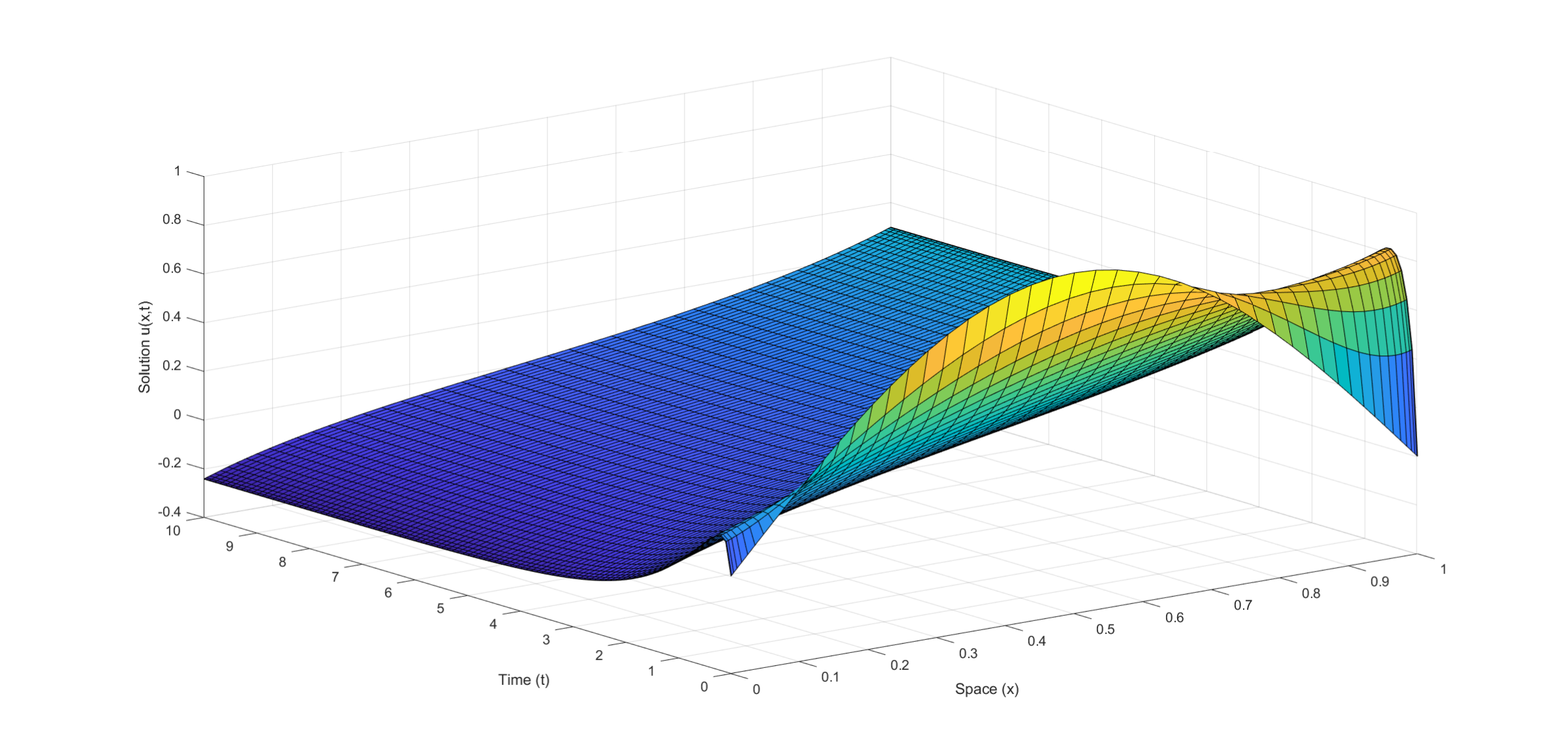}}
	\caption{Time evolution of uncontrolled GBH equation for $\nu = 0.1, \alpha=1,\beta=1,\delta=1,\gamma=0.5$  and  $u_{0}(x) = \sin(\pi x)$.}
	\label{Fig1}
\end{figure}

Figure \ref{Fig1} illustrates the evolution of the solution \( u(x,t) \) over time for the uncontrolled case, with parameters \(\alpha=1\), \(\beta=1\), \(\gamma=0.5\), \(\delta = 1\), \(\nu =0.1 \), and the initial condition \( u(x,0) = \sin(\pi x) \). When a non-adaptive control law, as specified in the equation \eqref{eq6} is applied with \(\eta_{1} =1\) and \(\eta_{2} =1\), Figure \ref{fig2} demonstrates that \( u(x,t) \) approaches the desired zero state. This observation validates Theorem \ref{thm2.3}. 

\begin{figure}[h]
	\centerline{\includegraphics[width=16cm]{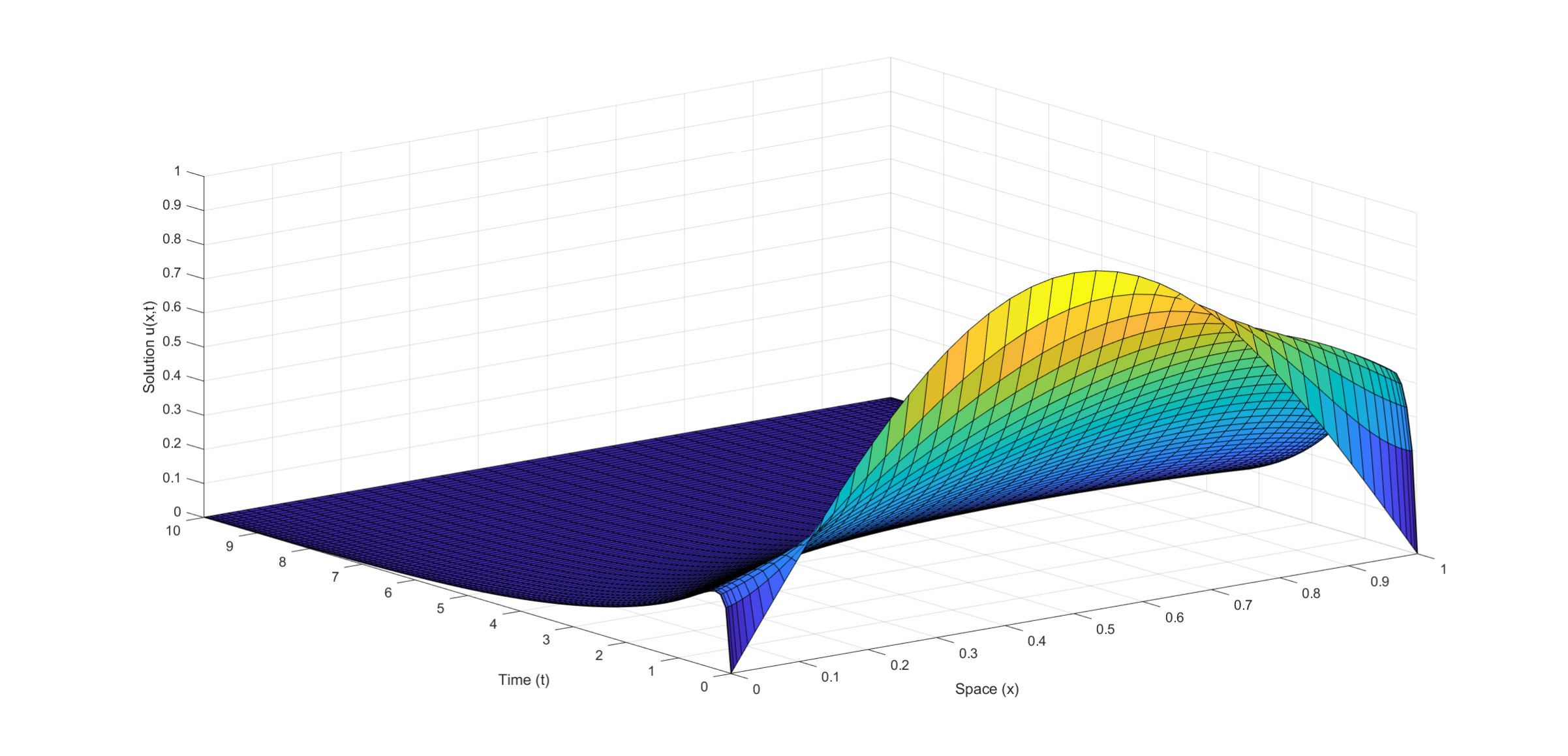}}
	\caption{Time evolution of controlled GBH equation for $\nu = 0.1, \alpha=1,\beta=1,\delta=1,\gamma=0.5,\eta_{1} = 1,\eta_{2} =1 $ and $u_{0}(x) = \sin(\pi x)$.}
	\label{fig2}
\end{figure}
\begin{figure}[h]
	\centerline{\includegraphics[width=16cm]{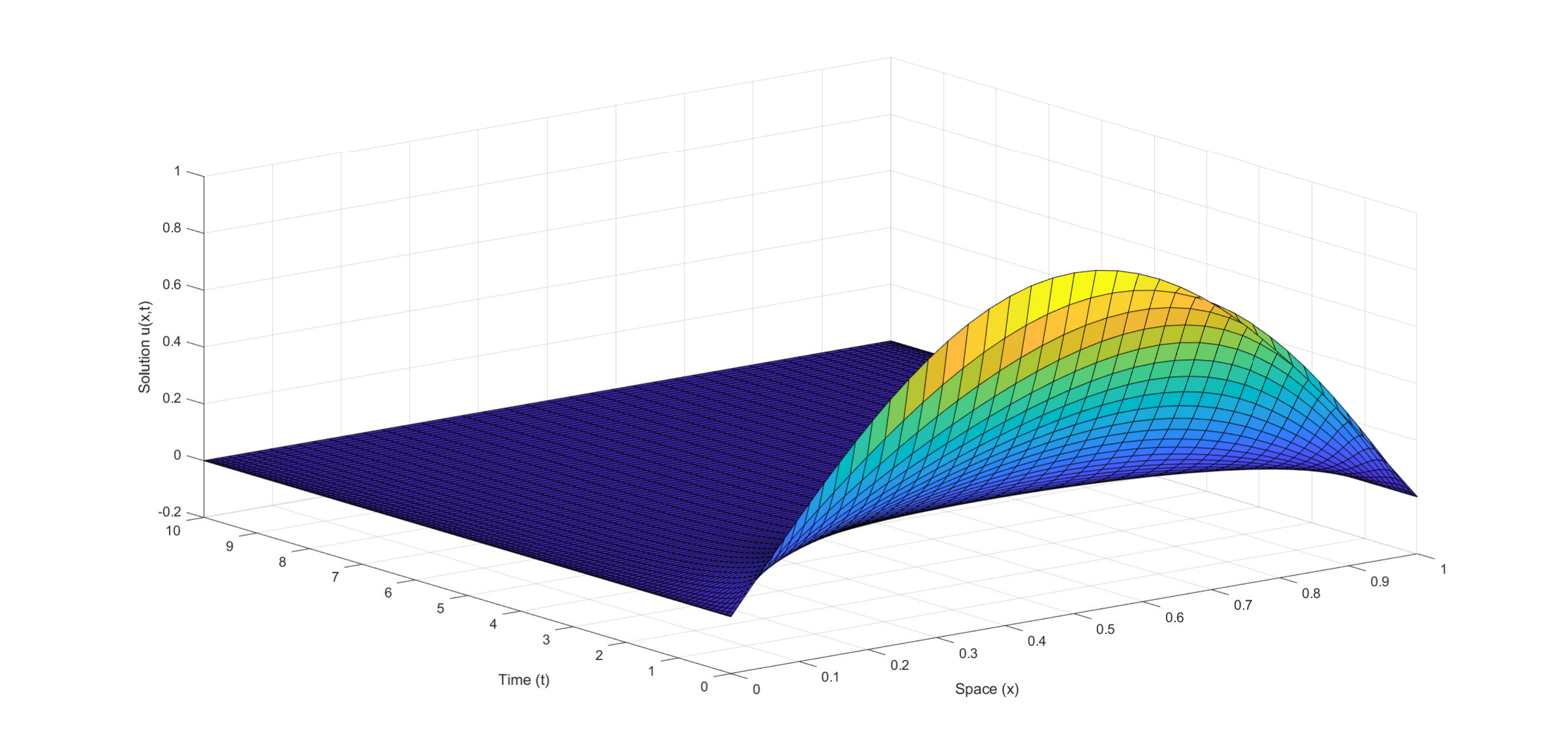}}
	\caption{Time evolution of adaptive controlled damped unforced GBH equation for $\nu = 0.1, \alpha=1,\beta=1,\delta=1,\gamma=0.5,\kappa = 0.05 $ and $u_{0}(x) =\ sin(\pi x)$.}
	\label{fig3}
\end{figure}

Figure \ref{fig3} presents the time evolution of the solution to damped GBH equation under the application of adaptive control with \(\kappa=0.05\). The numerical results correspond to this setting, clearly shows the decay of the solution to zero state, in agreement with the analytical results proved in Theorem \ref{thm2.5}.

Figures \ref{fig4}, \ref{fig5}, and \ref{fig6} further verify the exponential stability for both the non-adaptive and adaptive cases in comparison with the uncontrolled equation. These figures are presented under the parameters \(\alpha =2\), \(\beta =2\), and \(\delta = 3\), with all other conditions remains unchanged.

\begin{figure}[h]
	\centerline{\includegraphics[width=16cm]{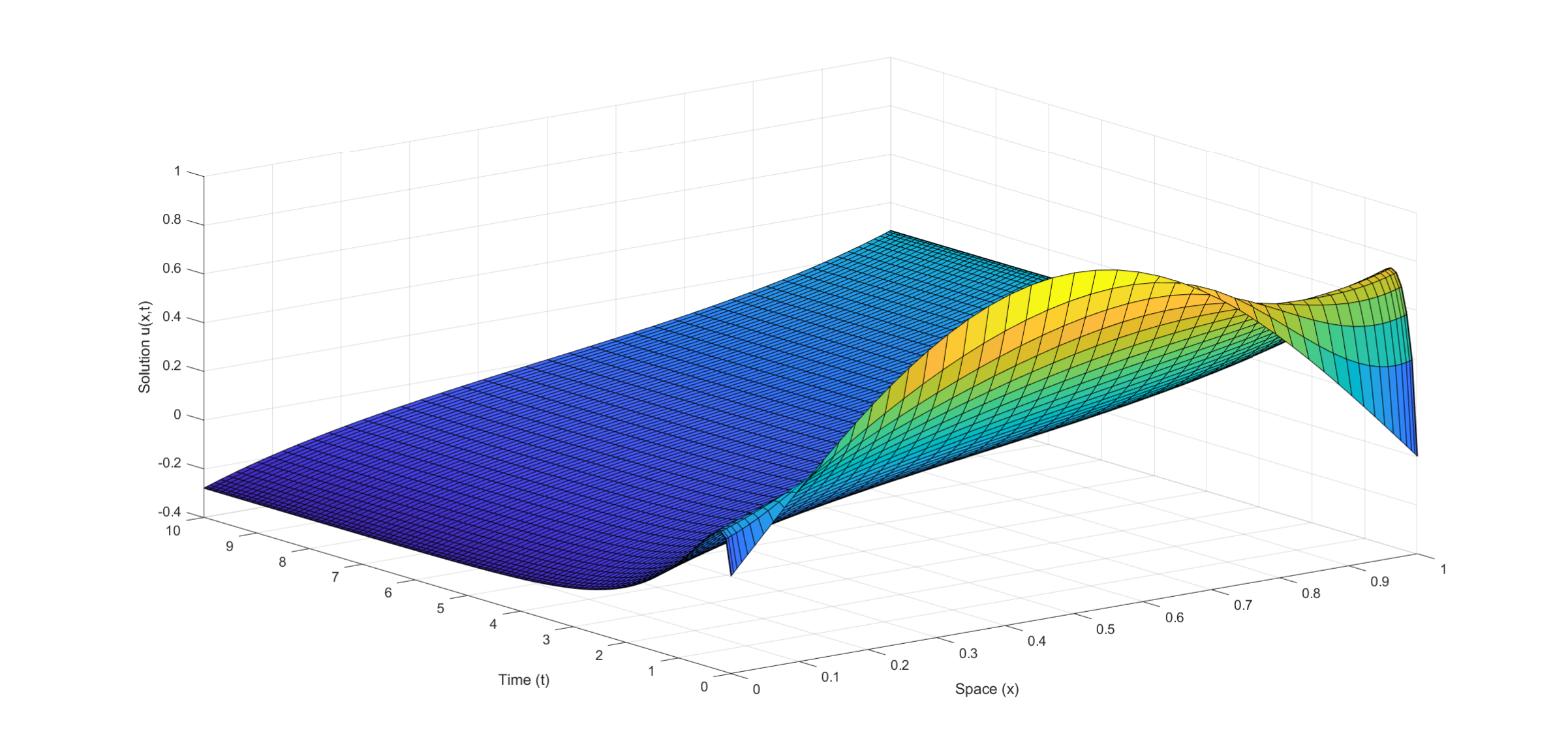}}
	\caption{Time evolution of uncontrolled GBH equation for $\nu = 0.1, \alpha=2,\beta=2,\delta=3,\gamma=0.5$  and $ u_{0}(x) = \sin(\pi x)$}
	\label{fig4}
\end{figure}
\begin{figure}[h]
	\centerline{\includegraphics[width=16cm]{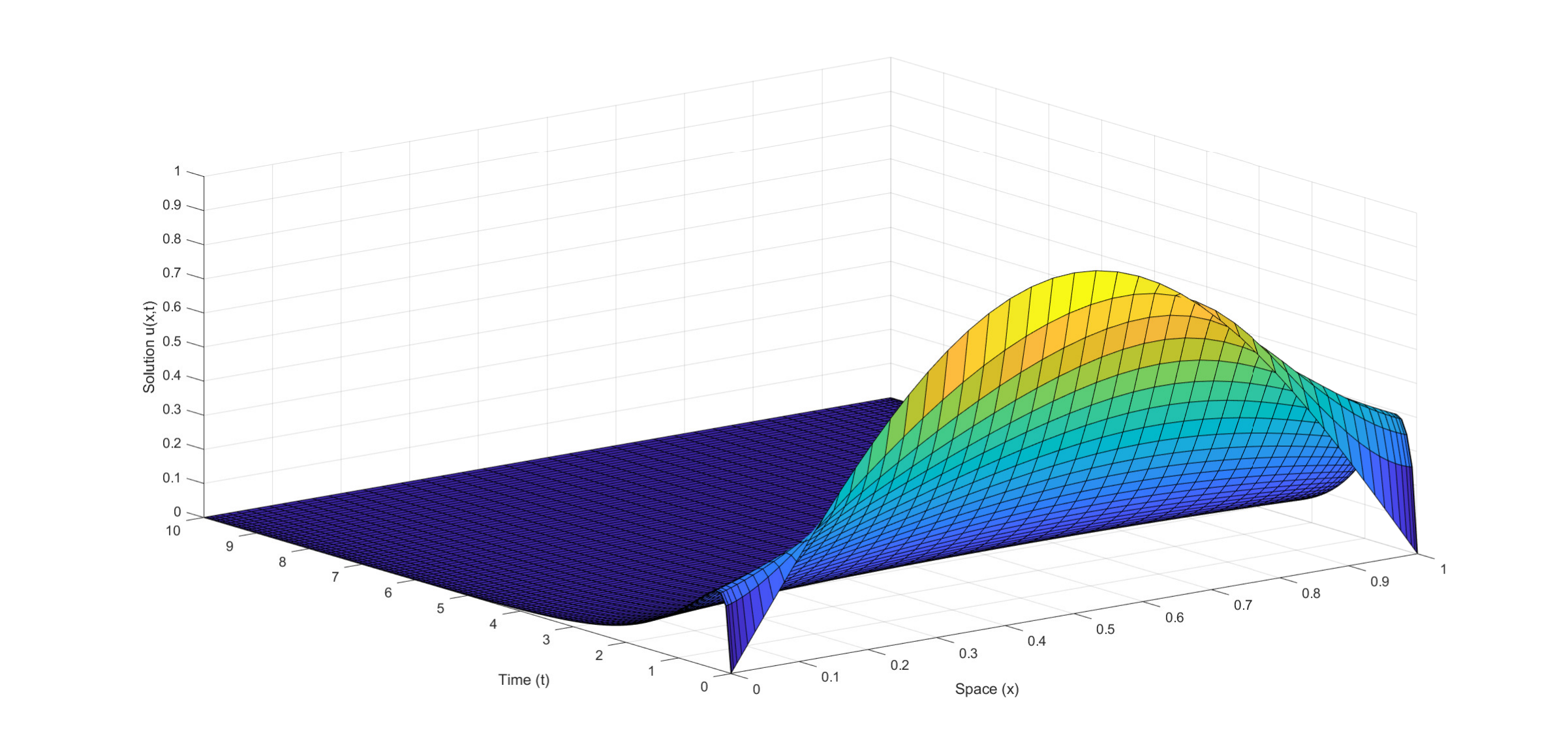}}
	\caption{Time evolution of controlled GBH equation for $\nu = 0.1, \alpha=2,\beta=2,\delta=3,\gamma=0.5,\eta_{1} = 1,\eta_{2} =1$ and $u_{0}(x) = \sin(\pi x)$.}
	\label{fig5}
\end{figure}
\begin{figure}[h]
	\centerline{\includegraphics[width=16cm]{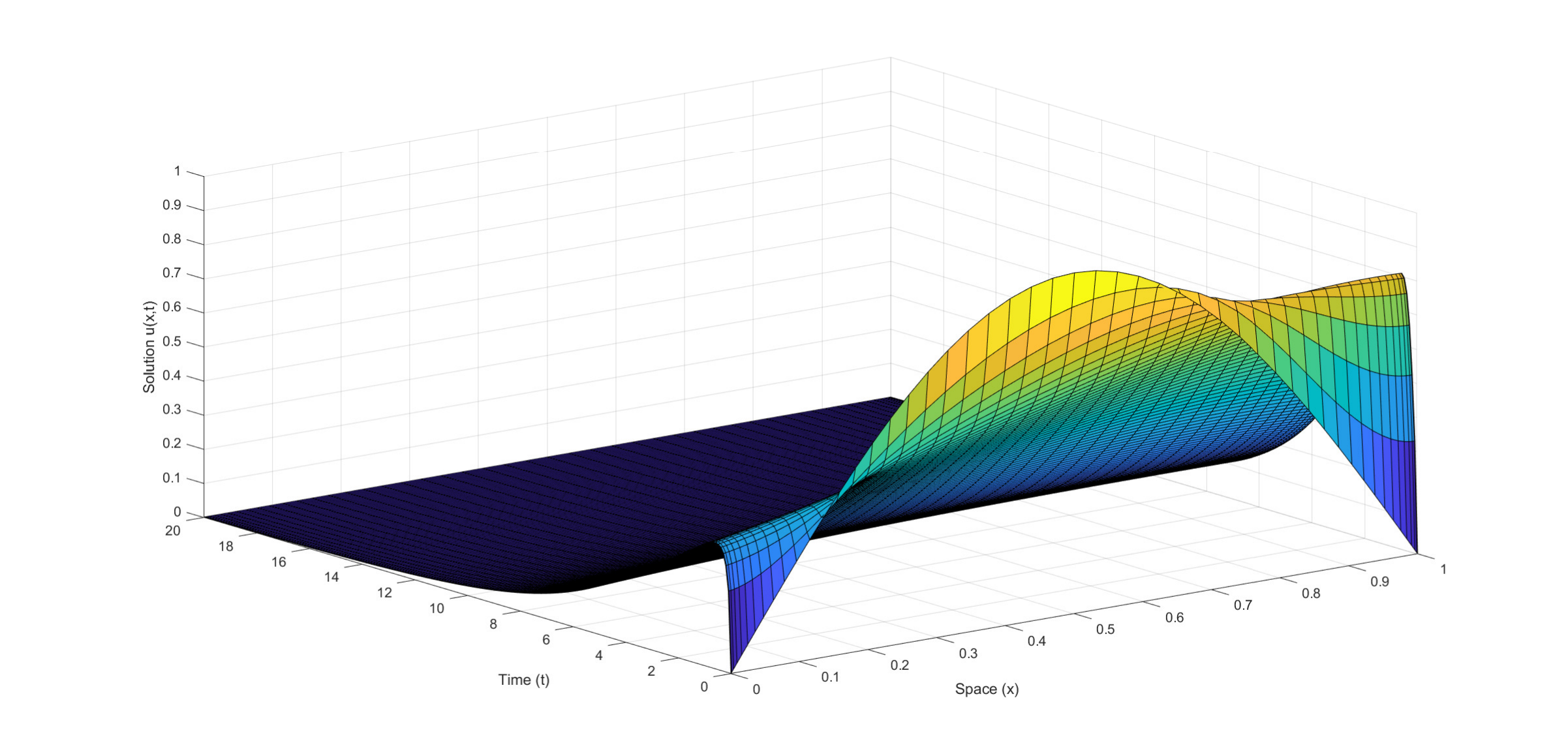}}
	\caption{Time evolution of adaptive controlled damped unforced GBH equation for $\nu = 0.1, \alpha=2,\beta=2,\delta=3,\gamma=0.5,\kappa = 0.05$ and $u_{0}(x) = \sin(\pi x).$}
	\label{fig6}
\end{figure}

	\end{document}